\documentclass[10pt]{amsart}
\usepackage{amsbsy}
\usepackage{graphicx,epsfig,subfigure,psfrag}
\begin{document}

\newtheorem{theorem}{Theorem}
\newtheorem{proposition}{Proposition}
\newtheorem{lemma}{Lemma}
\newtheorem{corollary}{Corollary}
\newtheorem{definition}{Definition}
\newtheorem{remark}{Remark}
\newcommand{\tex}{\textstyle}
\numberwithin{equation}{section} \numberwithin{theorem}{section}
\numberwithin{proposition}{section} \numberwithin{lemma}{section}
\numberwithin{corollary}{section}
\numberwithin{definition}{section} \numberwithin{remark}{section}
\newcommand{\ren}{\mathbb{R}^N}
\newcommand{\re}{\mathbb{R}}
\newcommand{\n}{\nabla}
\newcommand{\p}{\partial}
\newcommand{\iy}{\infty}
\newcommand{\pa}{\partial}
\newcommand{\fp}{\noindent}
\newcommand{\ms}{\medskip\vskip-.1cm}
\newcommand{\mpb}{\medskip}
\newcommand{\AAA}{{\bf A}}
\newcommand{\BB}{{\bf B}}
\newcommand{\CC}{{\bf C}}
\newcommand{\DD}{{\bf D}}
\newcommand{\EE}{{\bf E}}
\newcommand{\FF}{{\bf F}}
\newcommand{\GG}{{\bf G}}
\newcommand{\oo}{{\mathbf \omega}}
\newcommand{\Am}{{\bf A}_{2m}}
\newcommand{\CCC}{{\mathbf  C}}
\newcommand{\II}{{\mathrm{Im}}\,}
\newcommand{\RR}{{\mathrm{Re}}\,}
\newcommand{\eee}{{\mathrm  e}}
\newcommand{\LL}{L^2_\rho(\ren)}
\newcommand{\LLL}{L^2_{\rho^*}(\ren)}
\renewcommand{\a}{\alpha}
\renewcommand{\b}{\beta}
\newcommand{\g}{\gamma}
\newcommand{\G}{\Gamma}
\renewcommand{\d}{\delta}
\newcommand{\D}{\Delta}
\newcommand{\e}{\varepsilon}
\newcommand{\var}{\varphi}
\newcommand{\lll}{\l}
\renewcommand{\l}{\lambda}
\renewcommand{\o}{\omega}
\renewcommand{\O}{\Omega}
\newcommand{\s}{\sigma}
\renewcommand{\t}{\tau}
\renewcommand{\th}{\theta}
\newcommand{\z}{\zeta}
\newcommand{\wx}{\widetilde x}
\newcommand{\wt}{\widetilde t}
\newcommand{\noi}{\noindent}
\newcommand{\uu}{{\bf u}}
\newcommand{\xx}{{\bf x}}
\newcommand{\yy}{{\bf y}}
\newcommand{\zz}{{\bf z}}
\newcommand{\aaa}{{\bf a}}
\newcommand{\cc}{{\bf c}}
\newcommand{\jj}{{\bf j}}
\newcommand{\ggg}{{\bf g}}
\newcommand{\UU}{{\bf U}}
\newcommand{\YY}{{\bf Y}}
\newcommand{\HH}{{\bf H}}
\newcommand{\GGG}{{\bf G}}
\newcommand{\VV}{{\bf V}}
\newcommand{\ww}{{\bf w}}
\newcommand{\vv}{{\bf v}}
\newcommand{\hh}{{\bf h}}
\newcommand{\di}{{\rm div}\,}
\newcommand{\ii}{{\rm i}\,}
\newcommand{\inA}{\quad \mbox{in} \quad \ren \times \re_+}
\newcommand{\inB}{\quad \mbox{in} \quad}
\newcommand{\inC}{\quad \mbox{in} \quad \re \times \re_+}
\newcommand{\inD}{\quad \mbox{in} \quad \re}
\newcommand{\forA}{\quad \mbox{for} \quad}
\newcommand{\whereA}{,\quad \mbox{where} \quad}
\newcommand{\asA}{\quad \mbox{as} \quad}
\newcommand{\andA}{\quad \mbox{and} \quad}
\newcommand{\withA}{,\quad \mbox{with} \quad}
\newcommand{\orA}{,\quad \mbox{or} \quad}
\newcommand{\atA}{\quad \mbox{at} \quad}
\newcommand{\onA}{\quad \mbox{on} \quad}
\newcommand{\ef}{\eqref}
\newcommand{\mc}{\mathcal}
\newcommand{\mf}{\mathfrak}

\newcommand{\ssk}{\smallskip}
\newcommand{\LongA}{\quad \Longrightarrow \quad}
\def\com#1{\fbox{\parbox{6in}{\texttt{#1}}}}
\def\N{{\mathbb N}}
\def\A{{\cal A}}
\newcommand{\de}{\,d}
\newcommand{\eps}{\varepsilon}
\newcommand{\be}{\begin{equation}}
\newcommand{\ee}{\end{equation}}
\newcommand{\spt}{{\mbox spt}}
\newcommand{\ind}{{\mbox ind}}
\newcommand{\supp}{{\mbox supp}}
\newcommand{\dip}{\displaystyle}
\newcommand{\prt}{\partial}
\renewcommand{\theequation}{\thesection.\arabic{equation}}
\renewcommand{\baselinestretch}{1.1}
\newcommand{\Dm}{(-\D)^m}

\title
{\bf Countable families of solutions of a limit stationary
semilinear fourth-order  Cahn--Hilliard
 equation I. Mountain pass and Lusternik--Schnirel'man patterns in $\ren$}

\author{P.~\'Alvarez-Caudevilla, J.D.~Evans, and V.A.~Galaktionov}

\address{Universidad Carlos III de Madrid,
Av. Universidad 30, 28911-Legan\'es, Spain -- Work phone number:
+34-916249099} \email{pacaudev@math.uc3m.es}

\address{Department of Mathematical Sciences, University of Bath,
 Bath BA2 7AY, UK -- Work phone number: +44 (0)1225 386994 }
\email{masjde@bath.ac.uk}

\address{Department of Mathematical Sciences, University of Bath,
 Bath BA2 7AY, UK -- Work phone number: +44 (0)1225826988}
\email{masvg@bath.ac.uk}


\thanks{This work has been partially supported by the Ministry of Economy and Competitiveness of
Spain under research project MTM2012-33258.}


\date{\today}





\begin{abstract}
Solutions of the stationary semilinear Cahn--Hilliard equation
 $$
 -\Delta^2 u  - u -\Delta(|u|^{p-1}u)=0 \inB \ren, \quad \mbox{with} \quad p>1,
 $$
which are exponentially decaying at infinity, are studied. Using
the Mounting Pass Lemma allows us the determination of two different solutions. On
the other hand, the application of Lusternik--Schnirel'man (L--S) Category
Theory
shows the existence of, at least, a countable family of solutions.

However, through numerical methods it is shown that  the
whole set of solutions, even in 1D, is much wider. This suggests that,
actually,  there exists, at least, a countable set of countable
families of solutions, in which only the first one can be
obtained by the L--S min-max approach.

\end{abstract}

\maketitle




\section{Introduction and motivation for main problems}
 \label{S1}

\subsection{Models and preliminaries}

\noindent
This paper studies some multiplicity properties of the steady-states of
the following {\em fourth-order parabolic equation of the
Cahn--Hilliard (C--H) type}:
\begin{equation}
\label{11}
    \tex{u_{t} =-\Delta^2 u  - u -\Delta(|u|^{p-1}u)\, \quad
    \hbox{in}\quad \re^N\times \re_+ \whereA p>1.}
\end{equation}
The problem is completed with bounded
smooth initial data,
\be
\label{13}
    \tex{
    u(x,0)=u_0(x) \quad \hbox{in} \quad \ren
    .}
\ee
 Assuming that data $u_0(x)$ are sufficiently fast exponentially
 decaying at infinity, the same behaviour holds for the unique classic solution
of \ef{11}, at least locally in time, since, for $p>1$, $u(x,t)$
may blow-up in finite time; see key references and results in
\cite{PV} and \cite{EGW1}.

\ssk

The classic \emph{Cahn--Hilliard equation}  describes
the dynamics of pattern formation in phase transition in alloys,
glasses, and polymer solutions.
When a binary solution is cooled
sufficiently, phase separation may occur and then proceed in two
ways: either nucleation, in which nuclei of the second phase
appear randomly and grow, or, in the so-called spinodal
decomposition, the whole solution appears to nucleate at once and
then periodic or semi-periodic structures appear. Pattern
formation resulting from phase transition has been observed in
alloys, glasses, and polymer solutions. Parallelly, these types of equations
possesses a great interest in biology and, after certain transformations (see below)
as prototype for nerve conduction in the form of a FitzHugh--Nagumo  system
 (cf. \cite{FiguMit,KlassMit}).

\emph{Cahn--Hilliard equations-type} such as \eqref{11} have been studied by many authors during the last
decades. Among some of the works where several aspects related to these equations are \cite{AP, NC, UnCh, NC08, ANC1}.
Additionally, one can check the surveys in \cite{EGW1,GVSur02}, where necessary aspects
 of global existence and blow-up of solutions for \ef{11} are
 discussed in sufficient detail.

From the mathematical point of view, the non-stationary equation
\ef{11} involves a fourth order elliptic operator and it contains
a negative viscosity term. In a more general (and applied)
setting, the unknown function is a scalar $u=u(x,t)$, $x\in \ren$,
$t\in\re_+$, and the equation reads
\begin{equation}
\label{ch1}
    u_{t} - \D K(u)=0\quad \hbox{in} \quad \ren \times \re_+\,
    \whereA
    K(u):= -\nu \D u + f(u), \quad \nu >0.
\end{equation}
Then, in general, the function $f(u)$ is a polynomial of the order
$2p-1$,
\begin{equation*}
 \tex{
    f(u):= \sum_{j=1}^{2p-1} a_j u^j, \quad p\in \N,\;\; p\geq 2.
    }
\end{equation*}
In particular, the classic Cahn--Hilliard equation corresponds to
the case $p=2$ and
\begin{equation*}
    f(u):= -\eta u+ \mu u^3, \quad \eta,\mu >0.
\end{equation*}

This equation has been extensively
studied in the past years but many questions still remain
unanswered, especially in relation to multiplicity problems.


\subsection{Variational approach and main results}


Thus,  we concentrate on the analysis of multiplicity results
for the fourth-order elliptic stationary \emph{Cahn--Hilliard}
equation
\begin{equation}
\label{maineq}
    \tex{-\Delta^2 u - u - \Delta (|u|^{p-1}u)=0\, \quad
    \hbox{in}\quad \re^N \whereA p>1.}
\end{equation}
Note that \eqref{maineq} is not variational in $L^2(\re^N)$,
though it is variational in $H^{-1}(\re^N)$. Hence, multiplying
\eqref{maineq} by $(-\D)^{-1}$,
 we obtain an elliptic equation with a non-local
 operator of the form
  \be
 \label{nonloc}
    -\D u+(-\D)^{-1} u -|u|^{p-1} u=0.
\ee
 Here, as customary,
$(-\D)^{-1}u=v$, if
 $$
 -\D v=u \,\,\, \mbox{in} \,\,\, \re^N,
  \quad v(x) \to 0 \asA
 x \to \iy.
 $$
This yields the following $C^1$-functional associated with
\eqref{nonloc}:
\begin{equation}
\label{funn}
    \tex{ \mc{F}(u):=\frac{1}{2} \int\limits_{\re^N} |\nabla u|^2 + \frac{1}{2} \int\limits_{\re^N} |(-\D)^{-\frac 1 2} u|^2 -
     \frac{1}{p+1} \int\limits_{\re^N}
    |u|^{p+1},
    }
\end{equation}
Solutions of the equation \eqref{nonloc} are then obtained as
critical points of the functional \eqref{funn}.

The non-local operator $(-\D)^{-1}$ is  a positive linear integral
operator from $L^q(\re^N)$ to $L^2(\re^N)$, within the Sobolev's range, i.e., $1\leq q\leq \frac{2N}{N-2}$. 
Moreover, the operator $(-\D)^{-\frac 1 2}$
can be correctly  defined as the square root of the operator
$(-\D)^{-1}$ and it will also be referred to as a non-local linear operator.

Also, since the problem is set in $\re^N$ we are defining this
operator in a class of exponentially decaying functions; see below
the details and conditions to have the
 weak expression of the problem with these exponentially decay solutions.

A similar fourth-order
 problem was studied in
  \cite{PV}, where further references can be found.
 However, in \cite{PV} the stationary equation \eqref{maineq} was considered
 in a bounded
 smooth
 domain $\O \in \ren$, with  homogeneous Navier-type boundary conditions 
 \begin{equation}
\label{12}
   \tex{u=\D u=0 \quad \mbox{on} \quad \partial \O.}
\end{equation}
In particular, it was shown that this
 problem \eqref{maineq}, \ef{12}  admits a countable  family of solutions (critical points/values).
  Moreover, in \cite[\S~6]{GMPSob2} for a different variational
  fourth-order problem in $\re$, it was shown that a wider
  countable set
  of countable families of
 solutions can be expected, where only the first infinite family is the Lusternik--Schnirel'man one.

  Hence, due to the analysis
 performed in \cite{GMPSob,GMPSob2} we need to study, in a general multidimensional
 geometry, existence and multiplicity for the elliptic equation \eqref{maineq}
 in a class of functions properly decaying at infinity (in fact exponentially),
\be
\label{infcon}
    \tex{\lim_{|x| \rightarrow \infty} u(x)=0.}
\ee
Consequently, to get the results obtained here one can proceed following two different approaches.

Firstly, bearing in mind solutions satisfying \ef{infcon}, with a
fast decay at infinity, one can choose a sufficiently large radius $R>0$
for the ball $B_R$ and consider the variational problem for the
functional $\mc{F}(u)$ denoted by \eqref{funn} in $W_0^{1,2}$ and
assuming Dirichlet boundary conditions on $S_R=\p B_R$. Thus, both
spaces $L^2(B_R)$ and $L^{p+1}(B_R)$ are compactly imbedded into
$W_0^{1,2}(B_R)$ in the subcritical Sobolev's range
 \be
  \label{sub1}
   \tex{
    1<p \le p_*= 1 + \frac {N+2}{N-2} \big(\equiv \frac {2N}{N-2}=p_{\rm S}\big), \quad N \ge 3
     \quad (p_*=+\iy\,\,\,\mbox{for}\,\,\,N=1,\,2).
     }
 \ee
In other words,
\be
\label{comem}
     \tex{W_0^{1,2}(B_R) =H_0^1(B_R)  \hookrightarrow L^{p+1}(B_R) \Leftrightarrow 1<p < p_*.}
\ee Note that, here, for the fourth-order elliptic operator in
\ef{maineq}, $p_{\rm S} =\frac {2N}{N-2}$, with $N\geq 3$,
because this operator has the representation
 $$
  \tex{-\D^2
u-\D(|u|^{p-1}u)=-\D(\D u+ |u|^{p-1}u),}
 $$
 so that, the necessary embedding features are governed by a standard second-order one
 $$\D u+|u|^{p-1}u.$$
The next step is passing to the limit as $R \to +\iy$, by using some
uniform (in $R$) bounds on such families of solutions in $B_R$.
Since the category of the functional subset increases without
bound in such a limit, eventually, we then expect to arrive at, at
least, a countable family of various critical values/points.

However, here we perform a variational study
directly in $\ren$, which was done previously for many
fourth-order ODEs and elliptic equations; see \cite{KKV, KKVV, PT}
as key examples (though those equations, mainly, contain coercive
operators, with ``non-oscillatory" behaviour at infinity). 

Namely,
by a linearized analysis we first check that equation
\ef{maineq} provides us with a sufficient ``amount" of exponentially
decaying solutions at infinity. Obviously, in any bounded class of
such functions, and in a natural functional viewing, since, loosely
speaking, nothing happens at infinity (effectively, the solutions
vanish there), the variational problem can be treated as the one
in a bounded domain. So that the embedding in \ef{comem} comes in
charge in some sense.

\ssk

\begin{remark}\label{rem:compact}
Throughout this paper we shall consider radial solutions.
Hence, we take into consideration a well known result  about continuous  Sobolev's embedding (see for instance \cite{adams}),
\be
\label{compact}
\tex{E \hookrightarrow L^{p}(\re^N),\quad\hbox{with}\quad
1<p \le p_*,}
\ee
which are compact replacing $E$ by the radial subspace $W_{\rm rad}^{2,2}(\re^N)$ and  if in addition $2\le N$ and $p<p^*$ (see \cite{Lions-JFA82})
where
\be
\label{sobo12}
\tex{p_*=1 + \frac {N+4}{N-4} = \frac {2N}{N-4} =p_{\rm S}, \quad N \ge 5 \quad \hbox{and}\quad
     \quad p_*=+\iy\,\,\,\mbox{for}\,\,\,N=1,\, 2,\, 3,\, 4.}
     \ee
     Hence, there is a constant $S_N>0$ such that
$$\|u\|_{L^{p+1}(\re^N)} \leq S_N \|u\|_{W_{\rm rad}^{2,2}(\re^N)}.$$
Note that for the second order case one has the continuous embedding \eqref{compact} in the subcritical Sobolev's range \eqref{sub1}.
\end{remark}

\vspace{0.2cm}

Let us briefly summarize what we obtain here.
 First, we perform an analysis based on the application of the Mountain
Pass Theorem in order to ascertain the existence of one solution and, additionally, the existence of more than one 
for the equation \eqref{maineq}. To ascertain the existence of a second solution we use the
auxiliary equation
\be
 \label{nonlocMin}
    -\D u+(-\D)^{-1} u - |u+u^*|^{p-1} |u+u^*| +|u^*|^{p-1} u^*=0,
\ee
 where $u^*$ represents a solution of the problem  \eqref{nonloc}. Note that, since $u^*$ is a solution of \eqref{nonloc} if 
 $u$ is a solution of \eqref{nonlocMin} then, $u+u^*$ will be also a solution of \eqref{nonloc}.

  Thus, by construction and proving the existence of a solution for the equation
  \eqref{nonlocMin},
 we find the existence of a second solution for \eqref{nonloc} applying again a Mountain Pass argument.
Indeed, performing this analysis and thanks to the equivalence between
\eqref{maineq} and \eqref{nonloc}, we finally obtain the existence
of at least a second solution for our main equation
\eqref{maineq}.

Thus, we state (proved in Section\;\ref{SM}) the following.

\begin{theorem}
Suppose $N\geq 2$ and radial solutions in $W_{\rm rad}^{1,2}(\re^N)$ with exponential decay. 
Then, the non-local equation \eqref{nonloc} possesses at least two solutions in $W_{\rm rad}^{1,2}(\re^N)$.
\end{theorem}

 As far as we know, to get the existence of more than one solution, this seems to be the best
  available approach, since,
  for this type of
 higher-order PDEs we have a big lack of classical methodology and PDE theory.

In addition, secondly, we apply a L--S-fibering approach to get
a countable family of solutions (critical points), though without
any detailed information about how they look like. 

Therefore we state, and prove in Section\;\ref{S4}, the following.

\begin{theorem}
Suppose $N\geq 2$ and radial solutions in $W_{\rm rad}^{1,2}(\re^N)$ with exponential decay. Then, 
there is a countable family of solutions for  the non-local equation \eqref{nonloc} of the L--S type.
\end{theorem}

Finally, we apply
advanced numerical methods to describe general ``geometric"
structure of various solutions assuming symmetric for even profiles and anti-symmetric conditions for odd profiles (see below). 
In particular, we introduce some chaotic patterns for equation 
\eqref{maineq} for $p=3$ and $p=2$, showing some profiles that become very chaotic away from the point of symmetry.

These numerical experiments suggest that  the
whole set of solutions, even in 1D, is much wider.  However, these numerical experiments, together with
some analytical approaches and estimates, will be extended and analysed with more detail in \cite{AEGnegII}. Indeed, 
it will be proved there that there exists, at least, a countable set of countable
families of critical points, in which only the first one can be
obtained by the L--S min-max approach.

Also, we observe that performing numerical experiments, 
shooting smoothly from $x=0$ with $u'(0)=u''(0)=u'''(0)=0$ and varying $u(0)$, 
those 
chaotic patterns become more periodic when $u(0)$ increases. Indeed, this fact appears to be sooner for 
$p=3$ than for $p=2$.

Moreover, it should be mentioned that this kind of transition behaviour is seen in similar phase solidification fourth-order equations, 
such as the Kuromoto-Sivashinsky and Swift-Hohenberg equations \cite{PE}, \cite{CG}, as a critical order parameter increases.

\vspace{0.2cm}


\subsection{Previous related results}


Recall again that, in \cite{PV},   existence and multiplicity
results were obtained for the steady-states of the unstable C--H
equation of the form
\be
\label{gameq}
\tex{ -\Delta^2 u + \g u - \Delta (|u|^{p-1}u)=0 \quad
    \hbox{in}\quad \O \quad (p>1),}
 \ee
 for a real parameter $\g$, where the multiplicity essentially
 depended on this parameter, which affected the category of the
 functional subset associated with the principle linear operator.
The analysis was based on variational methods such as the {\em
fibering method}, potential operator theory and
Lusternik--Schnirel'man category-genus theory, and others, such as
homotopy approaches or perturbation theory existence.
Specifically, it was obtained that, depending on
the parameter $\g$, there exists a different number of stationary
solutions, i.e.,
\begin{itemize}
\item If the parameter $\g\leq  K \l_1$,  with $K>0$ a positive constant and $\l_1>0$ the first eigenvalue
of the bi-harmonic operator, i.e., $\D^2 \varphi_1=\l_1
\varphi_1$ with Navier boundary conditions \eqref{12}, then there exists at least one  solution for the
equation \eqref{gameq}, and;

\item When the parameter is greater than the  first eigenvalue of the bi-harmonic
equation $\l_1$, multiplied by the positive constant $K$, then
there will not be any solution at all, if one assumes only
positive solutions. However,
 for oscillatory solutions of changing sign the number of possible solutions increases with the value of the parameter $\g$. In fact, when the parameter
 $\g$ goes to infinity, one has an arbitrarily large number of distinct solutions. 
\end{itemize}

\begin{remark}
Note that the previous distinction in terms of the number of solutions is related to pattern formation of problems such as the Kuramoto-Sivashinsky equation; 
see \cite{Hoy,PE} for further references and details.
\end{remark}

Furthermore, the so-called Mountain Pass Theorem
\cite{AmRa} has been previously applied to find a second solution. Indeed, in
\cite{SatWat} the authors analyzed the existence of a second
solution for the fourth-order elliptic problem
\be
\label{sinpro}
    \left\{\begin{array}{c}
    \tex{ -\Delta^2 u -c_1\D u +c_2 u= u^{p} + k \sum_{i=1}^{m} \a_i \d_{a_i} \quad
    \hbox{in}\quad \mc{D}'(\re^N),}
 \ssk
    \\
    \tex{u(x)>0,\quad u(x)\to 0\quad \hbox{as}\quad |x|\to 0.}
    \end{array}\right.
 \ee
 where $\d_{a_i} $ represents the delta function supported at $a_i\in\re^N$ and
 $$\tex{1<p<\frac{N}{N-4}, \quad c_1^2-4c_2\geq 0,}$$
 and with $N\geq 5$, $k>0$, $m\in\N$, $\a_i>0$, $c_1,c_2>0$.

 Other methods utilized to solve similar problems to \eqref{maineq} such as
 \be
 \label{smp}
 \tex{\D^2 u+\b^2 \D u+u= f(u),}
 \ee
 might be Hamiltonian Methods or the Strong Maximum Principle. The previous equation can be written by
 $$\tex{(\D-\mu_1)(\D-\mu_1)u= f(u),}$$
 where $\mu_1$ and $\mu_2$ are the squares of the roots of the characteristic polynomial
 $$
 \tex{\mu^2+\b^2\mu+1=0,\quad \hbox{given by}\quad \mu_i=\frac{1}{2}\big(-\b^2\pm \sqrt{\b^4-4}\big),\; i=1,2.}
   $$
 Then, if $\b\geq \sqrt{2}$ both roots are real and positive and we can write the equation \eqref{smp} as the system
 \be
\label{sysmp}
    \left\{\begin{array}{l}
    \tex{ \Delta u -\mu_1 u = v,}\\
    \tex{\D v -\mu_2 v =f(u),}
    \end{array}\right.
 \ee
 and apply the Strong Maximum Principle having a positive solution
 $$u(x)>0 \quad \hbox{for}\quad x\in\re^N,$$
 if we assume that $u$ is a homoclinic orbit (see \cite{PT,SanjWei} for any further details with $N=1$).
 Moreover, in this case one can ascertain certain  qualitative properties for that homoclinic orbit, such as the existence of a
 precisely critical point for $u$, a priori
 estimates and the symmetry of $u$ with respect to that critical
 point. However, if $0<\b<\sqrt{2}$, one cannot obtain such a decoupling.

 Also, similar fourth-order equations to \eqref{smp} have been analyzed in \cite{KKV,KKVV}  in the dimensional space
 $\re^4$ and via Hamiltonian methods obtaining the connection between the critical points of the Hamiltonian.

\ssk

 On the other hand, note that problem \eqref{maineq} can be written as the following elliptic
 system:
 \be
 \label{noncosys}
    \left\{\begin{array}{l}
    \tex{ -\Delta u = -v-|u|^{p-1}u,}\\
    \tex{-\D v =u,}
    \end{array}\right.\quad \hbox{in}\quad \re^N,
 \ee
 where
 $\tex{v:=(-\D)^{-1}u,}$
 which gives a different perspective to the problem in hand.
 The solutions of the system \eqref{noncosys} represent the steady-states of a reaction-diffusion system of the form
 \be
 \label{parncsys}
    \left\{\begin{array}{l}
    \tex{u_t -\Delta u = -v-|u|^{p-1}u,}\\
    \tex{v_t-\D v =u,}
    \end{array}\right.\quad \hbox{in}\quad \re^N,
 \ee
 with a great interest in biology and as prototype for nerve conduction in the form of a FitzHugh--Nagumo
  system
 (cf. \cite{FiguMit,KlassMit}). Mathematically, we note
 that system \eqref{noncosys} is weakly coupled.
  However, this coupling is non-cooperative, in other words, the coupled terms have different signs.
 This means that the Maximum Principle is not valid for the system \eqref{noncosys}. For a cooperative
 system where the Maximum Principle can be applied to a very similar system to \eqref{noncosys}
 see \cite{PaNon}.
 
 Basically, so far, one can see that,
 due to the lack of comparison methods,
  Maximum Principle, and several others classical methods in the analysis of higher-order PDEs. Therefore, in general the methodology
 is very limited and restricted to very specific examples.


\subsection{Further extensions}


Our results can be applied to 
other C--H models. For example note
that, for the ``true" fourth-order semilinear operator, the
critical Sobolev's exponent is different:
 \be
 \label{4.1}
 \tex{
 -\D^2 u + |u|^{p-1}u=0 \LongA p_S= \frac {N+4}{N-4} \quad
 \mbox{for} \quad N \ge 5.
 }
 \ee
Obviously, a critical Sobolev's range as in \ef{4.1} occurs for a
different sixth-order C--H equation
 \be
 \label{4.2}
 \D^3 u \pm u - \D(|u|^{p-1}u)=0 \inB \ren.
  \ee
  However, if the unstable nonlinear diffusion operator is of fourth
  order, as in
\be
 \label{4.3}
 \D^3 u \pm u + \D^2(|u|^{p-1}u)=0 \inB \ren,
  \ee
we again arrive at the ``second-order" Sobolev range as in
\ef{sub1}.

Equations \ef{4.2} and \ef{4.3} can be studied in similar lines,
but some aspects become more technical, though not affecting the
principal conclusions and results.



\section{Preliminary results: exponentially decaying patterns in $\ren$}
\label{S2.Exp}



\subsection{Exponentially decaying patterns in $\ren$}


 The preliminary conclusions  presented here formally allow us
to consider our equations \ef{maineq} in the whole $\ren$, unlike
as in \cite{PV} where the problem was assumed to be
in a bounded domain $\O \subset \ren$.

Indeed, for the functional
\ef{funn} we deal with the integrals over $\ren$ and, actually, with the
functional setting over a certain weighted Sobolev
space\footnote{This is just a characteristic of a functional
class: surely, we cannot use any weighted metric, where any
potential approaches are lost.}, instead of $W_0^{2,2}(\O)$
as assumed in \cite{PV}. Such a functional setting of the problem in
$\ren$ is  key in what follows. In fact, a proper functional
setting assumes certain admissible asymptotic decay of solutions
at infinity, which, for \ef{maineq}, is governed by the
corresponding linearized operator.

Thus, considering \ef{maineq} in the radial geometry, with
$u=u(r)$ and $r = |x| \ge 0$, we then obtain
\be
 \label{inf1}
 \begin{split}
  \tex{ \D^2 u \equiv  u^{(4)} + \frac{2(N-1)}{r} u''' } & \tex{ 
    + \frac{2(N-1)(N-3)}{r^2} u'' - \frac{(N-1)(N-3)}{r^3} u' }\\ & \tex{=- u-p |u|^{p-3} u \left( u u''  + (p-1) \left( u' \right)^2
        + \frac{(N-1)}{r} u u' \right)\, ,
 }
 \end{split}
 \ee
Next, as usual, calculating the admissible decaying asymptotics
from \ef{inf1}, using a two scale WKBJ-type asymptotics, as a first
approximation (sufficient for our purposes), we use an
exponentially pattern of the form $ u(r) \approx r ^\d \,\eee^{a r}$ (as $r\to \infty$) in \eqref{inf1} leading easily to the
following characteristic equation:
 \be
 \label{inf111}
  \tex{
a^4=-1 \andA \d= - \frac{N-1}2.
 }
  \ee
  To be precise, note that the first equation in \eqref{inf111} comes from the homogeneity of the 
  leading terms in \eqref{inf1}. The second equality in \eqref{inf111} comes from a similar argument after 
  evaluating  the next leading terms on the left-hand side in \eqref{inf1}.  
 This
 yields a {\em two-dimensional} exponential bundle:
 \be
 \label{inf1N}
  \tex{
 u(r)\approx r^{-\frac{N-1}2} {\mathrm e}^{-r / \sqrt 2} \big[
C_1 \cos\big( \frac {r}{\sqrt 2} \big)+ C_2 \sin\big( \frac
{r}{\sqrt 2} \big)\big]\, ,}
 \ee
 where $C_{1,2} \in \re$ are two arbitrary parameters of this linearized
 bundle.
 
 Through this asymptotic analysis we shall be able to show some patterns after performing a shooting problem in Section\;\ref{S5}. 

Let us now perform a preliminary analysis. In this radial
geometry, any regular bounded solution of \ef{maineq} must satisfy
{\em two} boundary conditions at the origin (making the
bi-Laplacian non-singular)
\be
   \label{bc1}
   u'(0)=u'''(0)=0.
    \ee
   Hence, using a standard shooting strategy from $r = +\iy$, algebraically, at least two parameters are needed to
    satisfy both \ef{bc1}.

Looking again at \ef{inf1N}, where there exist {\em two}
parameters $C_{1,2} \in \re$, we observe that
 matching with two symmetry boundary conditions \ef{bc1} yields a
 well-posed and well-balanced algebraic ``2D--2D shooting problem".

To justify existence of such solutions (critical points), we then
need carefully apply different variational techniques. Of course,
for the elliptic problem in $\ren$, one needs a more technical and
delicate ``asymptotic separation" analysis. Namely, one needs to
resolve the following ``separation" procedure (as is well-known, for the
bi-Laplacian, a standard separation of variables is not
available):
 \be
 \label{sep.1}
  \tex{
  \D^2 u \equiv \big(\D_r+ \frac 1{r^2} \D_\s)^2 u=-u+...\,,
  }
  \ee
  where $\D_\s$ is the Laplace--Beltrami operator on $S^{N-1}$.
  Indeed, using the representation
   \be
   \label{sep.2}
   u(x)= R(r)Y(\s)+... \, ,
   \ee
   with $R(r)$ obtained above,
 it then leads to a complicated equations on $Y(\s)$.
In general, such an asymptotic separation procedure is expected to
determine a sufficiently wide infinite-dimensional asymptotic
bundle of solutions with an exponential decay at infinity.

However, we do not need such a full and rather technical analysis.
We must admit that we still do not know that whether a countable
family of L--S critical points are radially symmetric solutions or
not. If the former is true, then the above radial analysis is
sufficient. In general, using our previous experience, we expect
that min-max critical points are not all radially symmetric, but
cannot prove that.

Finally, we note that, obviously, this is not the case in 1D.
Indeed, in addition to the symmetric (even) profiles
satisfying \ef{bc1}, there exist others dipole-like/anti-symmetric (odd) solutions such that
 \be
 \label{bc111}
 u(0)=u''(0)=0.
  \ee
We can also check existence of such solutions numerically; see Section\;\ref{S5}.


\subsection{Spectral theory in $\ren$}

Now, we consider the linear spectral problem for the
corresponding linearized
non-local equation (\ref{nonloc}), i.e.,
\be
\label{linCH}
 \tex{
{\bf L} \psi_\b \equiv -\D \psi_\b + (-\D)^{-1} \psi_\b= \l_\b
\psi_\b \,\,\,\,\mbox{in}
 \,\,\,\, \re^N, \quad \lim_{|x|\to \infty}\psi_\b(x)=0,
 }
 \ee
  which was analyzed, however, in a bounded domain, in full detailed in \cite{FiguMit}. In problem
  \eqref{linCH} $\l_\b$ represents the $\b$th-eigenvalue associated with the eigenfunction $\psi_\b$.

 In the following we show and prove several properties of the spectrum of the linear eigenvalue
problem (\ref{linCH}). Subsequently, we will apply them, in particular, in order to get estimations of the category; see section\;\ref{S4}.

As usual, we begin checking that the linearized equation
(\ref{linCH}) admits solutions with a proper exponential decay at
infinity. Indeed, writing it down in the equivalent form
 \be
 \label{ppp1}
 \left\{ \begin{array}{ll} \D^2 \psi_\b+\psi_\b=-\l_\b \D \psi_\b &
\hbox{in}\quad \re^N,
 \ssk
\\ \quad \hbox{and} & \lim_{|x|\to +\infty}
\psi_\b(x)=0, \end{array}\right.
  \ee
 we have that a proper exponential decay at infinity holds (here, as before, $r=|x| \gg 1$, so we are again restricted to
 the radial setting of this linear spectral problem).

Moreover, due to the positivity of the operator ${\bf L}=\D^2+{\rm Id}$ on the left-hand side
of (\ref{ppp1}), $\l_\b=0$ is not an eigenvalue, so that
$\l_\b>0$, for any $\b$. Secondly, owing to \cite[Theorem\;$VI.8$]{Bre} and the compactness of the inverse integral
operator the spectrum might contain either infinitely many isolated real eigenvalues or a finite number of isolated eigenvalues,
formed by a monotone sequence of eigenvalues
  $$
  0 < \l_1\leq \l_2\leq \ldots \leq \l_\b\leq \ldots.
  $$
In other words, for a sufficiently large $\alpha$, the resolvent
of the operator ${\bf L}+\alpha {\rm I}$ is compact and, hence,
the spectrum is discrete and, of course, real, by symmetry
(self-adjoint). Note as well that when the operator is
self-adjoint the method shown in \cite{Hess} can be used to prove
that there are infinitely many eigenvalues. Thus, we have the
following result:

\begin{proposition}
\label{marce}
The operator ${\bf L}$ admits a discrete set of eigenvalues that tend to $+\infty$ and there exists at least a solution
$\psi \in W_{\rm rad}^{1,2}(\re^N)$.
\end{proposition}

 \vspace{0.2cm}

\noi {\bf Remarks.} {\rm
\begin{itemize}
\item The strict positivity of eigenvalues, for eigenfunctions with an
  exponential decay at infinity follows from the equality
 \be
 \label{eq88}
  \tex{
  \int|\D \psi_\b|^2 + \int |\psi_\b|^2= \l_\b \int |\n \psi_\b|^2,
  }
   \ee
   obtained by multiplying \ref{ppp1} by $\psi_\b$ and integrating in $L^2$.
  \item Moreover, the corresponding associated family of eigenfunctions $\{\psi_\b\}$ is a complete orthogonal set in $L^2$.
  \item According to our analysis above, to get exponentially decaying solutions,
the real eigenvalues in (\ref{linCH}) must satisfy
  \be
  \label{sub12}
  \l_\b>0, \quad \hbox{for any $\b$}.
   \ee
   Note that problem \eqref{ppp1} admits a positive first eigenvalue $\l_1$ characterized by the Raileigh quotient
\be
\label{rai}
\tex{\l_1:=\inf_{u\in  W_{\rm rad}^{2,2}(\re^N)}  \frac{ \int_{\re^N} |\D u|^2+\int_{\re^N} u^2}{\int_{\re^N} |\nabla u|^2}.}
\ee
\end{itemize}
   }

  \vspace{0.2cm}

    Moreover,
   we also find the following simple observation:

 \begin{proposition}
 \label{Pr.NNN}
 If $\l_\b>0$ is an eigenvalue of (\ref{ppp1}) for any $\b$, then
  \be
  \label{lambda3}
   \l_\b \ge 2.
   \ee
   \end{proposition}

   \noi{\em Proof.} Indeed, from \ref{eq88}, integrating by parts and applying the H\"{o}lder's inequality yields
    \be
    \label{eq191}
    \tex{
    \big| \int\n \psi_\b \cdot \n \psi_\b \big| =\big| \int \D \psi_\b \psi_\b\big| \le
    \big(\int |\D \psi_\b|^2 \big)^{\frac 12} \big( \int \psi_\b^2
    \big)^{\frac 12} \le \frac 12 \big[ \int (|\D \psi_\b|^2 +
    \psi_\b^2)\big],
     }
     \ee
     which proves the proposition.
     \qed

\vspace{0.2cm}

 Furthermore, the natural space for the eigenfunctions of problem \eqref{ppp1} is $W_{\rm rad}^{2,2}(\re^N)$,
i.e., the closure of
$C_0^{\infty}$-functions with respect to the norm
 $$
 \tex{
 \|u\|_{W^{2,2}(\re^N)}^2 = \int_{\re^N} |\D
u|^2+ \int_{\re^N} |\nabla u|^2+ \int_{\re^N} u^2,
 }
  $$
   and the associated inner product
    $$
     \tex{
    \left\langle u,
v \right\rangle =\int_{\re^N} \D u \D v +   \int_{\re^N} \nabla u \nabla v+ \int_{\re^N} u\,v.
 }
 $$
Indeed, the space $W_{\rm rad}^{2,2}(\re^N)$ is a
reflexive Banach space.

\vspace{0.2cm}

Under those assumptions we have the following variational
expression of the problem \eqref{ppp1}:
  $$\tex{  \int_{\re^N} \D
\psi_\b \D v +\int_{\re^N} \psi_\b v =\l_\b \int_{\re^N} \nabla \psi_\b \cdot \nabla v
\quad \hbox{for any}\quad v\in W_{\rm rad}^{2,2}(\re^N),\quad \hbox{for any $\b$}.}$$
Thus, $\psi_\b \in
W_{\rm rad}^{2,2}(\re^N)\setminus\{0\}$ is an eigenfunction of the problem
\eqref{ppp1} associated with the eigenvalue $\l_\b$. Moreover, a
weak formulation of the problem requires the following result:

\begin{lemma}
\label{leexpin} Suppose the $u\in W_{\rm rad}^{2,2}(\re^N)$. Then,
there is a sequence $\{R_n\}\subset \re_+$, with $R_n\to \infty$
as $n\to +\infty$, such that $$\tex{ \lim_{n\to +\infty}  \int_{\p
B_{R_n}}  u \frac{\p u}{\p {\bf n}} {\mathrm d}S=0,\; \lim_{n\to
+\infty}  \int_{\p B_{R_n}} \nabla u \frac{\p \nabla u}{\p {\bf
n}} {\mathrm d}S=0 \; \hbox{and}\; \lim_{n\to +\infty}  \int_{\p
B_{R_n}}  u \frac{\p \D u}{\p {\bf n}} {\mathrm d}S=0},
  $$
  where $B_{R_n}$ is
the ball of radius $R_n>0$, centered at the origin, and ${\bf n}$
denotes the unitary outward normal vector on $\partial B_{R_n}$.
\end{lemma}

Note that, since, for $\l_\b >0$ for any $\b$,
 we always deal with
solutions with exponential decay at infinity (see above), the
previous assumptions are, hence, achievable.


\section{Mountain Pass Theorem and existence of at least two solutions}
\label{SM}


In this section, we apply the celebrated   Mountain Pass Theorem
(cf. \cite{AmRa,Ra2} for details of this highly cited theorem) to
ascertain the existence of a solution and, consequently, of at least a second solution,
for the problem \eqref{maineq}
in $\re^N$.

Recall that, in $\ren$, with any $N \ge 2$, we are restricted to a
class of radially symmetric solutions, for which we know their
exponential decay at infinity. For $N=1$, we can deal with both
even and odd patterns. However, in general, restrictions to both
symmetries or not, makes no difference in the variational
analysis. Nevertheless, in order to have the compact Sobolev's embedding \eqref{compact}
in the subcritical range \eqref{sub1} our results in this section are restricted to $N\geq 2$.


\subsection{Mountain Pass Theorem to ascertain the existence of a solution for the problem \eqref{maineq}}


To obtain the existence of solutions for the stationary
Cahn--Hilliard equation \eqref{maineq} we apply the Mountain Pass
Theorem to the equation \eqref{nonloc}. Hence, we look for
critical points of the functional $\mc{F}(u)$ \eqref{funn} which
correspond to weak solutions of the equation \eqref{nonloc}, i.e.,
\begin{equation}
 \label{defweak11}
    \tex{\int\limits_{\re^N} \nabla u \cdot \nabla \varphi + \int\limits_{\re^N} (-\D)^{-\frac 1 2} u \cdot
    (-\Delta)^{-\frac 1 2} \varphi  - \int\limits_{\re^N} |u|^{p}  \varphi=0,  }
\end{equation}
for any $\varphi \in W_{{\rm rad}}^{1,2}(\re^N)\equiv  H_{{\rm rad}}^1(\re^N)$ (or $C_0^\infty(\re^N)$) a radial function.
However, by elliptic regularity, those weak solutions $u$ are also strong
solutions of the equation \eqref{nonloc}; see \cite{Berger}. 

Then, we look for the existence of critical points for
the functional $\mc{F}(u)$ \eqref{funn}.
As previously proved in \cite{PV} this functional is weakly lower semicontinuous and its
Fr\'echet derivative has the expression
$$\tex{D_{u}\mc{F}(u) \varphi
    :=\int\limits_{\re^N} \nabla u \cdot \nabla \varphi + \int\limits_{\re^N} (-\D)^{-\frac 1 2}
    u\cdot
    (-\Delta)^{-\frac 1 2} \varphi  - \int\limits_{\re^N} |u|^{p} \varphi,}
     \quad \varphi \in H_{{\rm rad}}^1(\re^N),
     $$
such that the directional derivative (Gateaux's
derivative) of the functional \eqref{funn} is the following
 $$   \tex{\frac{\mathrm d}{{\mathrm d}t}\, \mc{F}(u+t\varphi)_{|t=0}= \left\langle D_u \mc{F}(u), \varphi\right\rangle
    =   D_{u}\mc{F}(u) \varphi.}
$$
Moreover, we know that the functional is $C^1$  and
the critical points of the
functional \eqref{funn} denoted by
\begin{equation*}
    \mc{C}:=\{u \in W_{{\rm rad}}^{1,2}(\re^N)\,:\,
    D_{u}\mc{F}(u) \varphi=0,\quad \hbox{for any}\quad \varphi \in W_{{\rm rad}}^{1,2}(\re^N)\}
\end{equation*}
are weak solutions in $H_{{\rm rad}}^1(\re^N)$ for the equation
\eqref{nonlocM}, i.e.,
  $$  \tex{D_{u}\mc{F}(u) \varphi=0.}
$$
Thus,
$u\in \mc{C}$ if and only if
 $$
  \tex{\mc{F}'(u):=\int\limits_{\re^N} |\nabla u|^2 + \int\limits_{\re^N} \big|(-\D)^{-\frac 1 2} u\big|^2
   - \int\limits_{\re^N}  |u|^{p+1}=0, }
      $$
      where $\mc{F}'$ is called the gradient of $\mc{F}$ at $u$.
Again, by classic elliptic regularity
(Schauder's theory; see \cite{Berger} for further details) we
will then always obtain classical solutions for such equations.

The main ingredient we are applying in
getting the existence of a solution is the celebrated
Mountain Pass Theorem due to Ambrosetti--Rabinowitz \cite{AmRa}.
Before stating the Mountain Pass Theorem, we show a very important
and necessary condition in order to satisfy the theorem.
This is the so-called \emph{Palais--Smale condition}.

\subsection{Palais--Smale condition (PS)}

 Let $E$ be a Banach
space and $\{u_n\} \subset E$ a sequence such that
\be
\label{conPS}
\mc{F}(u_n)
\quad\hbox{is bounded and}\quad  \mc{F}'(u_n) \to 0,\quad \hbox{as}\quad n\to \infty,
\ee
then
$\{u_n\}$ is pre-compact, i.e., $\{u_n\}$ has a convergence
subsequence. In particular, $$\mc{F}(u_n) \to c\quad
\hbox{and}\quad \mc{F}'(u_n) \to 0 \Rightarrow \hbox{$\{u_n\}$ has
a convergence subsequence.}$$ The (PS) condition is a convenient
way of imposing some kind of compactness into the functional
$\mc{F}$. Indeed, this (PS) condition implies that the set of
critical points at the level value $c$
$$\mc{C}_c=\{u\in E\,;\,
\mc{F}(u)=c,\;    \mc{F}'(u) =0\},$$ is compact for any $c\in\re$.

\begin{remark}
Since the functional $\mc{F}$ is $C^1$ it is
easily proved that if there exists a minimizing sequence $\{u_n\}$
of solutions of the equation \eqref{nonloc} weakly convergent in
$H_{{\rm rad}}^1(\re^N)$ to certain $u_0\in H_{{\rm rad}}^1(\re^N)$ and such that
$\mc{F}'(u_n) \to 0$ then we can assure that $u_0$ is a critical
point, i.e., $\mc{F}'(u_0)=0$.
\end{remark}

\begin{theorem}$(${\bf Mountain Pass Theorem}$)$
\label{MPT} Let $E$ (for example $H_{{\rm rad}}^1(\re^N)=W_{{\rm rad}}^{1,2}(\re^N)$) be a
Banach space.  Let the functional $\mc{F} \in C^1(E,\re)$ satisfy
the Palais--Smale condition. Moreover, suppose that $\mc{F}(0)=0$
and
\begin{enumerate}
\item[a)] There exist $\rho, \a >0$ such that
$$\mc{F}_{\p B(0,\rho)} \geq \a,$$
where $B(0,\rho)$ represents the ball centered at the origin and of radius $\rho>0$;
\item[b)] also, there exists $e\in E\setminus B(0,\rho)$ such that $\mc{F}(e)\leq 0.$
\end{enumerate}
Then, the functional $\mc{F}$ possesses a critical value $c\geq \a >0$ characterized as
\be
\label{minmaxc}
c=\inf_{g\in\G} \max_{\theta \in [0,1]} \mc{F}(g(\theta)), \quad \hbox{where}\quad \G:=\{g\in C([0,1],E)\,;\, g(0)=0,\; g(1)=e\}.
\ee
\end{theorem}

\begin{remark}
As Ambrosetti--Rabinowitz mentioned, on a heuristic level, the
theorem says that, if a pair of points in the graph of $\mc{F}$
are separated by a mountain range there must be a mountain pass
containing a critical point between them. Also, although the
statement of the theorem does not imply it, normally in the
applications the origin $u=0$ is a local minimum for the
functional $\mc{F}$. As we will see below that is our case.

Most of the critical points will be maxima or minima. However, we cannot assure directly
that those critical points are global maxima or minima, we shall need to work a bit harder to obtain that.
\end{remark}

\begin{remark}
To ensure the Palais--Smale condition, we assume that condition
\eqref{compact} with \eqref{sub1} is satisfied and $N\geq 2$. Moreover, in view of the decay at
infinity of the solutions \eqref{infcon}, we can consider more
general nonlinearities $f(x,u)$ such as
  $$
  |f(x,s)|\leq
C(|s|^p+1),
 $$  because, in this case, by using \eqref{infcon},
 $$
 \tex{\frac{f(x,s)}{|s|^p} \longrightarrow 0\quad \hbox{as}\quad
s\to \infty;}
  $$
   However, we will focus on the nonlinearity
included in the equation \eqref{nonloc}.
\end{remark}

Thus, we first show that for our problem the trivial solution $u=0$ is a local minimum.
\begin{lemma}
\label{lezero}
The functional $\mc{F}(u)$ defined by \eqref{funn} possesses a local minimum at $u=0$.
\end{lemma}
\begin{proof}
Take a function $g\in H_{\rm rad}^1(\re^N)$ normalized in the following way
$$\tex{ \int\limits_{\re^N} |\nabla g|^2=1.}$$ 
Then, taking a real
number  $t$ sufficiently close to $0$, using the expression of the functional $\mc{F}$ and  applying the compact Sobolev's embedding \eqref{compact}, \eqref{sub1} 
we find that
$$
\tex{\mc{F}(tg)= \frac{t^2}{2}+ \frac{t^2}{2} \int\limits_{\re^N}
|(-\D)^{-\frac 1 2} g|^2
     -  \frac{t^{p+1}}{p+1} \int\limits_{\re^N}  |g|^{p+1} \geq \left(\frac{t^2}{2} -  \frac{t^{p+1} K}{p+1} \right),
  }
$$
for a constant $K>0$ that depends on the Sobolev's constant $S_N$. 
Then,
we arrive at $$  \tex{\mc{F}(tg) \geq  \left(\frac{t^2}{2} -  \frac{t^{p+1}K}{p+1} \right)>0=\mc{F}(0),
  }
$$
with
$t$ sufficiently close to zero, i.e.,
   $$\tex{t^{p-1}< \frac{p+1}{2K}.}$$
\end{proof}

As a first step we prove that the (PS) condition is satisfied by the functional $\mc{F}$ \eqref{funn}.

\begin{lemma}
\label{leMPT12} The functional $\mc{F}$ denoted by  \eqref{funn}
satisfies the {Palais--Smale condition}.
\end{lemma}

\begin{proof}
Let $\{u_n\}$ be a sequence such that $\mc{F}(u_n) \to c$
and $\mc{F}'(u_n) \to 0$. Then, since $ \mc{F}'(u_n) \to 0$, for any $\e>0$, there exists a
subsequence of
 $\{u_n\}$ (denoted again by $\{u_n\}$) such that
 $$\tex{ \big| \int\limits_{\re^N} \nabla u_n \cdot \nabla \varphi + \int\limits_{\re^N} (-\D)^{-\frac 1 2}
    u_n \cdot
    (-\Delta)^{-\frac 1 2} \varphi  - \int\limits_{\re^N} |u_n|^{p} \varphi\big| \leq \e \|\varphi\|_{W_{{\rm rad}}^{1,2}(\re^N)}.}
    $$
Indeed,  $\varphi=u_n$ yields
  \be
  \label{pls22}
  \tex{ \big| \int\limits_{\re^N}
|\nabla u_n|^2 + \int\limits_{\re^N} |(-\D)^{-\frac 1 2} u_n|^2 -
\int\limits_{\re^N}  |u_n|^{p+1}\big| \leq \e \|u_n\|_{W_{{\rm
rad}}^{1,2}(\re^N)}.}
   \ee
 Furthermore, due to \eqref{conPS} we have that
$\mc{F}(u_n)$ is bounded (or $\mc{F}(u_n)\to c$), i.e.,
 $$\tex{|\mc{F}(u_n)|\leq c,}$$
 \be
 \label{pls33}
  \tex{\hbox{or},\quad \frac{1}{2} \int\limits_{\re^N} |\nabla
u_n|^2 + \frac{1}{2} \int\limits_{\re^N} |(-\D)^{-\frac 1 2}
u_n|^2 -
     \frac{1}{p+1} \int\limits_{\re^N}     |u_n|^{p+1}=c+o(1),}
   \ee
 Hence, for a positive constant $\mu \in \re$ (to be chosen below) and using \eqref{pls22}, \eqref{pls33} 
 $$\tex{ \mu \mc{F}(u_n) - \left\langle D_u\mc{F}(u_n),u_n\right\rangle =c+ o(1)  \|u_n\|_{W_{{\rm
rad}}^{1,2}(\re^N)}= c+ \e  \|u_n\|_{W_{{\rm
rad}}^{1,2}(\re^N)}.}$$
Indeed, we actually have that
$$\tex{ \frac{\mu-2}{2} \int\limits_{\re^N} |\nabla u_n|^2 +  \frac{\mu-2}{2}
\int\limits_{\re^N} |(-\D)^{-\frac 1 2} u_n|^2 + \left(1- \frac{\mu}{p+1}\right) \int\limits_{\re^N} |u_n|^{p+1}= c+ \e  \|u_n\|_{W_{{\rm
rad}}^{1,2}(\re^N)}.}$$
Subsequently, for $\mu >2$ we arrive at the inequality
$$\tex{ \frac{\mu-2}{2}  \|u_n\|_{W_{{\rm rad}}^{1,2}(\re^N)}^2 +  \left(1- \frac{\mu}{p+1}\right) \int\limits_{\re^N} |u_n|^{p+1} -
\frac{\mu-2}{2} \int\limits_{\re^N} |u|^2 \leq c+ \e  \|u_n\|_{W_{{\rm
rad}}^{1,2}(\re^N)}.
}$$
Now, analysing the function
$$G(s)=  \left(1- \frac{\mu}{p+1}\right)  s^{p+1} - \frac{\mu-2}{2} s^2, \quad \hbox{for $s\geq 0$,}$$
and since $G(0)=0$ one can easily realize that
$$G(s)\geq 0,\quad \hbox{for any $s\geq 0$\; if \; $\mu <p+1$}.$$
Consequently, for any appropriate
$\e>0$ sufficiently small, and $2<\mu <p+1$
we get the boundedness of the norms in
$W_{{\rm rad}}^{1,2}(\re^N)$ for the elements of the subsequence
$\{u_n\}$. Note that $p>1$.

 Subsequently, due to the Hardy--Littlewood--Sobolev's inequality
 \be
 \label{hls}
    \tex{\big\|(-\D)^{-\frac 1 2} u_n\big\|_{L^2(\re^N)}\leq C \big\|u_n\big\|_{L^{p+1}(\re^N)}, \quad \hbox{with}\quad 1<p \le p_*},
\ee
we find that
$$|\mc{F}(u)|\leq \frac{1}{2} \|u\|_{W_{{\rm rad}}^{1,2}(\re^N)}^2 + K  \big\|u_n\big\|_{L^{p+1}(\re^N)}^2\quad \hbox{with $K$ a positive constant}.$$
Therefore, if $N\geq 2$ and $1<p < p_*$ for \eqref{sub1}, the Sobolev's embedding \eqref{compact}  is compact and, hence,
$$|\mc{F}(u)|\leq K  \|u\|_{W_{{\rm rad}}^{1,2}(\re^N)}^2,\quad \hbox{with $K$ a positive constant,}$$
obtaining the strong convergence of the subsequence $\{u_n\}$ and proving the (PS) condition.
\end{proof}

Finally, we apply the Mountain Pass
Theorem \eqref{MPT} in order to obtain the existence of solutions for the equation \eqref{nonloc}. Thus, we state the following result.

\begin{lemma}
\label{leMPT22}
Let the functional $\mc{F} \in C^1(W_{{\rm rad}}^{1,2}(\re^N),\re)$ be the functional denoted by \eqref{funn}. Then,
\begin{enumerate}
\item[a)] There exist $\rho, \a >0$ such that
$$\mc{F}_{\p B(0,\rho)} \geq \a,$$
where $B(0,\rho)$ represents the ball centered at the origin and of radius $\rho>0$;
\item[b)] also, there exists $e\in W^{1,2}(\re^N) \setminus B(0,\rho)$ such that $\mc{F}(e)\leq 0.$
\end{enumerate}
\end{lemma}

\begin{proof}
First, taking $u\in W_{{\rm rad}}^{1,2}(\re^N)\setminus \{0\}$, we
have that
     \be
     \label{mosob}
      \tex{
     \tex{\mc{F}(u)= K \|u\|_{W_{{\rm rad}}^{1,2}(\re^N)}^2 +
     o\big(\|u\|_{W_{{\rm rad}}^{1,2}(\re^N)}^2\big)\quad \hbox{as}\quad u\to 0.}
 }
 \ee
Indeed, applying the Hardy--Littlewood--Sobolev inequality \eqref{hls}
and the Sobolev's embedding \eqref{compact}, in the range
\eqref{sub1}, yields
\be
\label{pmas}
\tex{ \big\|u_n\big\|_{L^{p+1}(\re^N)}^2 \leq K \|u\|_{W_{\rm rad}^{1,2}(\re^N)}^2\quad \hbox{with $K$ a positive constant},}
\ee
 so that
 $$\tex{\int\limits_{\re^N} |u|^{p+1}= o\big(\|u\|_{W_{\rm rad}^{1,2}(\re^N)}^2\big) \quad \hbox{as}\quad u\to 0.}$$
Thus, we can assure that there exists $\rho, \a >0$ such that
$$\mc{F}_{\p B(0,\rho)} \geq \a.$$
Additionally,
assuming the compact Sobolev's embedding \eqref{compact}
and satisfying the decay condition at infinity \eqref{infcon},
  we find that
  \be
\label{minwfu}
    \begin{split} & \tex{\mc{F}(tu) =
    \frac{t^2}{2} \int\limits_{\re^N} |\nabla u|^2 + \frac{t^2}{2} \int\limits_{\re^N} |(-\D)^{-\frac 1 2} u|^2 -
     \frac{t^{p+1}}{p+1} \int\limits_{\re^N}
    |u|^{p+1}    } \\ & \tex{ \leq \frac{t^2K}{2} \|u\|^2_{W_{{\rm rad}}^{1,2}(\re^N)}  -\frac{t^{p+1}}{p+1}      \|u\|_{L^{p+1}(\re^N)}^{p+1}
   \to -\infty,
    }
    \end{split}
    \ee
as $t\to \infty$ and for a positive constant $K$ proving the
second condition of the lemma. Note that the positive constant $K$ comes after
  using the Hardy--Littlewood--Sobolev inequality \eqref{hls}
 with $p_*$ corresponding to the Sobolev range \eqref{sub1}.
\end{proof}

These results provide us with the existence of at least a solution for the non-local equation
\eqref{nonloc} and, hence, the existence of stationary solutions for the Cahn--Hilliard equation \eqref{maineq}.


\subsection{Existence of at least a second solution}


Now, we follow a similar argument line as performed above to get the existence of another solution for
the non-local equation \eqref{nonloc}.
Therefore, to ascertain the existence of at least  a second
solution, we will work on seeking for solutions of the equation
\be
 \label{nonlocM}
    -\D u+(-\D)^{-1} u - |u+u^*|^{p-1} |u+u^*| +|u^*|^{p-1} u^*=0.
\ee such that $u^*$ is solution of the equation \eqref{nonloc}
and, hence, to  the stationary Cahn--Hilliard equation
\eqref{maineq}. 

Note that, since $u^*$ is a solution of \eqref{nonloc} if 
 $u\neq 0$ is a solution of \eqref{nonlocM} then, $u+u^*$ will be also a solution of \eqref{nonloc} 
 different from $u^*$

As seen above the critical points of the functional $\mc{F}(u)$ \eqref{funn} correspond to weak solutions
of the equation \eqref{nonloc}.
However, by the elliptic regularity, $u^*$ is also a strong
solution of the equation \eqref{nonloc}. Now, for
\eqref{nonlocM}, we look for the existence of critical points for
the functional denoted by
\begin{equation}
\label{funnM}
\begin{matrix}
    \tex{ \mc{J}(u):=\frac{1}{2} \int\limits_{\re^N} |\nabla u|^2 +
    \frac{1}{2} \int\limits_{\re^N} |(-\D)^{-\frac 1 2} u|^2
    } \ssk \\     \tex{
    -
     \frac{1}{p+1} \int\limits_{\re^N}     |u+u^*|^{p+1}+  \frac{1}{p+1} \int\limits_{\re^N}
      |u^*|^{p+1} +  \int\limits_{\re^N}  u
      |u^*|^{p}.}
    \end{matrix}
\end{equation}
Furthermore, note that this functional is weakly lower
semicontinuous and its Fr\'echet derivative has the expression
$$\tex{D_{u}\mc{J}(u) \varphi = \int\limits_{\re^N} \nabla u \cdot
\nabla \varphi + \int\limits_{\re^N} (-\D)^{-\frac 1 2}
    u\cdot
    (-\Delta)^{-\frac 1 2} \varphi  - \int\limits_{\re^N}  \big(|u+u^*|^{p} -|u^*|^{p}\big)  \varphi,}
    $$
    such that the directional derivative of the functional \eqref{funnM} is
    \begin{equation}
\label{derv}
    \tex{\frac{\mathrm d}{{\mathrm d}t}\, \mc{J}(u+t\varphi)_{|t=0}= \left\langle D_u \mc{J}(u),\varphi\right\rangle
    =   D_{u}\mc{J}(u) \varphi.}
\end{equation}
 \begin{remark}
  To arrive at the expression for those derivatives and equalities above for the functional
 $\mc{J}(u)$ \eqref{funnM} follow
 previous arguments shown for the functional $\mc{F}(u)$ \eqref{funn} in \cite{PV}.
For several additional properties we stress
the work made in Sato--Watanabe \cite{SatWat}.  In particular,
Lemmas $5.1$ and $5.2$, where this argument to ascertain a second solution was used for a problem of the type \eqref{sinpro}.
\end{remark}

 Furthermore, due to \eqref{derv}, the critical points of the
functional \eqref{funnM} denoted by
\begin{equation*}
    \mc{C}_*:=\{u \in W_{{\rm rad}}^{1,2}(\re^N)\,:\,
    D_{u}\mc{J}(u) \varphi=0,\quad \hbox{for any}\quad \varphi \in W_{{\rm rad}}^{1,2}(\re^N)\}
\end{equation*}
are weak solutions in $H_{{\rm rad}}^1(\re^N)$ for the equation
\eqref{nonlocM}, i.e.,
   $$ \tex{\mc{J}'(u):=D_{u}\mc{J}(u) \varphi=0.}$$
Thus,
$u\in \mc{C}_*$ if and only if
\begin{equation}
\label{critp}
    \tex{\mc{J}'(u):=\int\limits_{\re^N} |\nabla u|^2 + \int\limits_{\re^N} \big|(-\D)^{-\frac 1 2} u\big|^2 - \int\limits_{\re^N}  \big(|u+u^*|^{p} -|u^*|^{p}\big) u=0.  }
\end{equation}
Again, by classic elliptic regularity the solutions of the equation \eqref{nonlocM} are also
classical solutions.

Then, as previously mentioned, we apply again Mountain Pass Theorem
to ascertain the existence of a second solution for the functional $\mc{J}$ denoted by \eqref{funnM}.

Furthermore, note that, again, since the functional $\mc{J}$ is $C^1$, if there exists a minimizing sequence $\{u_n\}$
of solutions of the equation \eqref{nonlocM} weakly convergent in
$H_{{\rm rad}}^1(\re^N)$ to certain $u_0\in H_{{\rm rad}}^1(\re^N)$, such that
$\mc{J}'(u_n) \to 0$ then we can assure that $u_0$ is a critical
point $\mc{J}'(u_0)=0$.

Also, since the nonlinearity
of the equation \eqref{nonloc} satisfies $$\tex{ |\xi
+u^*|^{p-1}|\xi+u^*|- |u^*|^{p-1} u^*=o(|\xi|)\quad \hbox{as}\quad
\xi\to 0,}$$ this implies that \eqref{nonlocM} possesses again the
trivial solution $u=0$. Indeed, it is not difficult to see that
$u=0$ is a local minimum just applying similarly Lemma\;\ref{lezero} since $\mc{J}(0)=0$.

Before proving the Mountain Pass Theorem for the perturbed problem \eqref{nonlocM}
with the associated functional \eqref{funnM} we prove some properties
for several auxiliary functions that will be used later.
\begin{lemma}
\label{leauxz}
Assume that $s\geq 0$ and $a>0$. Then the following holds:
\begin{enumerate}
\item[(i)] The function defined by
\be
\label{auxfun1}
\tex{T(s)=(a+s)^ps-a^p s-\frac{2+\mu}{p+1}(a+s)^{p+1} + \frac{2+\mu}{p+1}a^{p+1} + (2+\mu) a^ps +\frac{\mu}{2}pa^{p-1} s^2,}
\ee
satisfies that
$$\tex{T(s)\geq 0,\quad  \hbox{for any}\quad  s\geq 0 \quad \hbox{and}\quad \mu=\min\{1,p-1\}.}$$
\item[(ii)] Let
\be
\label{auxfun2}
\tex{H(a,s)= \frac{1}{p+1}  |a+s|^{p+1}-      \frac{1}{p+1}  |a|^{p+1} -   s |a|^{p}.}
\ee
Then, there exists $\e>0$
such that
\be
\label{auxsol}
 \tex{H(a,s) - \frac{p}{2}|a|^{p-1} s^2 \leq \e |a|^{p-1} s^2+ C_\e s^{p+1},}
     \ee
     with $C_\e$ a positive constant that depends on $\e>0$.
     \item[(iii)]
     \be
 \label{ineab}
 \tex{ |a+s|^{p} -|a|^{p}\leq C (|s|^p+|a|^{p-1} |s|),\quad \hbox{for a positive constant $C>0$,}}
 \ee
\end{enumerate}
\end{lemma}
\begin{proof}
\begin{enumerate}
\item[(i)] First we observe that
$$\tex{T(s)=T'(s)=T''(s)=0,}$$
and
$$\tex{T'''(s)=p(p-1) (p-2) (a+s)^{p-3}s+p(p-1)(a+s)^{p-2} (1-\mu).}$$
Therefore,
$$\tex{T'''(s)\geq 0, \quad \hbox{for any}\quad s\geq 0,\quad \hbox{if}\quad \mu=\min\{1,p-1\}.}$$
Hence, integrating three times we actually have that
$$\tex{T(s)\geq 0 \quad \hbox{for any}\quad s\geq 0,\quad \mu=\min\{1,p-1\}.}$$

\item[(ii)] This inequality is easily proved just using the next expression (integrating by parts)
    $$\tex{H(a,s)- \frac{p}{2}a^{p-1} s^2 =p(p-1)\int_0^s (a+\tau)^{p-2} (s-\tau)^2d\tau.}$$
    Indeed, applying Young's inequality
    yields
    $$
    \tex{p(p-1)\int_0^s (a+\tau)^{p-2} (s-\tau)^2d\tau
        \leq  C a^{p-2} s^3 +C s^{p+1}
    \leq\e a^{p-1} s^2 +C_\e s^{p+1},}
    $$
       if $s$ is very close to zero.
    \item[(iii)]      To prove inequality \eqref{ineab} we write
     \be
     \label{ineuse}
     \tex{ |a+s|^{p} -|a|^{p} = p\int_0^{s} |a+t|^{p-1}{\mathrm d}t.}
     \ee
Integrating and applying Newton's binomial yields
  $$
   \tex{
  p\int_0^{s}
|a+t|^{p-1}{\mathrm d}t =|a+s|^{p}- |a|^{p}=\sum_{k=0}^p
\binom{p}{k} |s|^k |a|^{p-k}\leq  C (|s|^p+|a|^{p-1} |s|),
 }
$$
 and, hence, the inequality holds for an appropriate constant $C$.
\end{enumerate}
\end{proof}

Subsequently, we prove the (PS) condition for functional $\mc{J}$ \eqref{funnM}.
\begin{lemma}
\label{leMPT1} The functional $\mc{J}$ denoted by  \eqref{funnM}
satisfies the {Palais--Smale condition} if $N\geq 2$ and $1<p<p^*$ for \eqref{sub1}.
\end{lemma}

\begin{proof}
Let $\{u_n\}$ be a sequence that satisfies the PS condition \eqref{conPS}.
Since $ \mc{J}'(u_n) \to 0$, for any $\e>0$, there exists a
subsequence of
 $\{u_n\}$ (denoted again by $\{u_n\}$) such that
 $$\tex{ \big| \int\limits_{\re^N} \nabla u_n \cdot \nabla \varphi + \int\limits_{\re^N} (-\D)^{-\frac 1 2}
    u_n \cdot
    (-\Delta)^{-\frac 1 2} \varphi  - \int\limits_{\re^N}  \big(|u_n+u^*|^{p} -|u^*|^{p}\big)
     \varphi\big| \leq \e \|\varphi\|_{W_{{\rm rad}}^{1,2}(\re^N)}.}
    $$
Indeed, if $\varphi=u_n$ yields $$\tex{ \big| \int\limits_{\re^N}
|\nabla u_n|^2 + \int\limits_{\re^N} |(-\D)^{-\frac 1 2} u_n|^2 -
\int\limits_{\re^N}  \big(|u_n+u^*|^{p} -|u^*|^{p}\big) u_n\big|
\leq \e \|u_n\|_{W_{{\rm rad}}^{1,2}(\re^N)}.}
    $$
    Additionally, due to the
boundedness of the functional $\mc{J}(u_n)$ we have that
$$\tex{|\mc{J}(u_n)|\leq c,}$$
\begin{align*}
\tex{\hbox{or},\quad \frac{1}{2} \int\limits_{\re^N} |\nabla u_n|^2 + \frac{1}{2}
\int\limits_{\re^N} |(-\D)^{-\frac 1 2} u_n|^2 } & \tex{-
     \frac{1}{p+1} \int\limits_{\re^N}     |u_n+u^*|^{p+1}+  \frac{1}{p+1}
     \int\limits_{\re^N}     |u^*|^{p+1} } \\ & \tex{+  \int\limits_{\re^N}      u_n |u^*|^{p}=c+o(1),}
     \end{align*}
Subsequently, for an appropriate positive constant $\mu\in \re$
$$
\tex{(2+\mu) \mc{J}(u_n) - \left\langle D_u \mc{J}(u_n),u_n\right\rangle
= c+ \e  \|u_n\|_{W_{{\rm rad}}^{1,2}(\re^N)},}$$
or equivalently
\be
\label{splon}
\begin{split}
\tex{\frac{\mu}{2}} & \tex{ \left(  \int\limits_{\re^N} |\nabla u_n|^2 + \int\limits_{\re^N} |(-\D)^{-\frac 1 2} u_n|^2 \right)
-  \frac{2+\mu}{p+1} \int\limits_{\re^N}     |u_n+u^*|^{p+1}+  \frac{2+\mu}{p+1}
     \int\limits_{\re^N}     |u^*|^{p+1} } \\ & \tex{ +  (2+\mu)\int\limits_{\re^N} u_n |u^*|^{p}
     +\int\limits_{\re^N}  \big(|u_n+u^*|^{p} -|u^*|^{p}\big) u_n\big|= c+ \e  \|u_n\|_{W_{{\rm rad}}^{1,2}(\re^N)}}
     \end{split}
     \ee
Now, we analyse the terms
$$\tex{\int\limits_{\re^N}  \big(|u_n+u^*|^{p} -|u^*|^{p}\big) u_n\big| - \frac{2+\mu}{p+1} \int\limits_{\re^N}     |u_n+u^*|^{p+1}+  \frac{2+\mu}{p+1}
     \int\limits_{\re^N}     |u^*|^{p+1} +  (2+\mu)\int\limits_{\re^N} u_n |u^*|^{p}.}
     $$
     Indeed, thanks to Lemma\;\ref{leauxz} (i) we find that
  \be
  \label{inlong}
  \begin{split}
  \tex{\int\limits_{\re^N}  \big(|u_n+u^*|^{p} -|u^*|^{p}\big) u_n\big| - \frac{2+\mu}{p+1} \int\limits_{\re^N}     |u_n+u^*|^{p+1}} & \tex{
  +  \frac{2+\mu}{p+1}
     \int\limits_{\re^N}     |u^*|^{p+1} +  (2+\mu)\int\limits_{\re^N} u_n |u^*|^{p}} \\ & \tex{\geq  \frac{\mu}{2} \int\limits_{\re^N} p |u^*|^{p-1} u_n^2,}
    \end{split}
    \ee
  and, hence, substituting \eqref{inlong} into \eqref{splon} yields to
$$
\tex{\frac{\mu}{2} \left(  \int\limits_{\re^N} |\nabla u_n|^2 + \int\limits_{\re^N} |(-\D)^{-\frac 1 2} u_n|^2-  \int\limits_{\re^N} p |u^*|^{p-1} u_n^2 \right)
= c+ \e  \|u_n\|_{W_{{\rm rad}}^{1,2}(\re^N)}}
    $$
  Next, we apply the spectral theory of the non-local weighted eigenvalue problem
\be
\label{nonlspe} \left\{ \begin{array}{ll} (-\D+(-\D)^{-1})
\psi_1 =\l_1^* a(x) \psi_1 & \hbox{in}\quad \re^N,\\ \quad
\hbox{and} & \lim_{|x|\to +\infty}
\psi(x)=0,\end{array}\right. \ee with a weight of the form
$a(x)=p|u^*|^{p-1}$, and where $\l_1^*$ is the first eigenvalue of the problem
\eqref{nonlspe} associated with the eigenfunction $\psi_1$.
Then, it is easy to prove (see \cite{FiguMit} for any further
details about the spectral theory of certain classes of non-local
operators of this form in bounded domains) that $$
\tex{\int\limits_{\re^N} |\nabla \psi|^2  + \int\limits_{\re^N} |(-\D)^{-\frac 1 2} \psi|^2 \geq \l_1^*
\int\limits_{\re^N} p |u^*|^{p-1} \psi^2,\quad \hbox{for any}
\quad \psi\in W^{1,2}(\re^N)}$$
Note that, to extend the spectral theory developed in \cite{FiguMit} for bounded domains, one needs to follow the analysis carried out above in Section\;\ref{S2.Exp}
   for these non-local problems, assuming exponential decay for the solutions.

Then, using the weighted eigenvalue problem \eqref{nonlspe}
we find that
$$
\tex{\frac{\mu}{2} \left(1-\frac{1}{\l_1^*}\right) \left(  \int\limits_{\re^N} |\nabla u_n|^2 + \int\limits_{\re^N} |(-\D)^{-\frac 1 2} u_n|^2 \right)
\leq c+ \e  \|u_n\|_{W_{{\rm rad}}^{1,2}(\re^N)}}
    $$
Note that $\mu >0$ and due to Proposition\;\ref{Pr.NNN} we can show that $\l_1^*>2$ so that
$$\tex{\frac{\mu}{2} \left(1-\frac{1}{\l_1^*}\right)>0,}$$
then,
$$
\tex{\frac{\mu}{2} \left(1-\frac{1}{\l_1^*}\right) \int\limits_{\re^N} |\nabla u_n|^2
\leq c+ \e  \|u_n\|_{W_{{\rm rad}}^{1,2}(\re^N)}}.
    $$
     Moreover,
   $$
\tex{\frac{\mu}{2} \left(1-\frac{1}{\l_1^*}\right)\|u_n\|_{W_{{\rm rad}}^{1,2}(\re^N)}^2 - \frac{\mu}{2} \left(1-\frac{1}{\l_1^*}\right) \int\limits_{\re^N} u_n^2
\leq c+ \e  \|u_n\|_{W_{{\rm rad}}^{1,2}(\re^N)}},
    $$
    and, hence, thanks to the Sobolev's embedding \eqref{compact} (see Talenti \cite{Tal} for best possible constants) and
    rearranging terms we arrive at
    $$
\tex{\frac{\mu}{2} \left(1-\frac{1}{\l_1^*}\right)K\|u_n\|_{W_{{\rm rad}}^{1,2}(\re^N)}^2
\leq c+ \e  \|u_n\|_{W_{{\rm rad}}^{1,2}(\re^N)}},
    $$
    for a positive constant $K$.
 Therefore, we finally get the boundedness of the norms in
$W_{{\rm rad}}^{1,2}(\re^N)$ for the elements of the subsequence
$\{u_n\}$.

To conclude the prove we show the strong convergence of the subsequence $\{u_n\}$. Indeed, applying the
Hardy-Littlewood-Sobolev's inequality \eqref{hls}, inequality \eqref{ineab} and the compact Sobolev's embedding \eqref{compact},
for $N\geq 2$ and $1<p<p^*$ for \eqref{sub1},
we find that
\begin{align*}
\tex{ |\mc{J}(u)| } & \tex{\leq \left|\frac{1}{2} \left(\int\limits_{\re^N} |\nabla u|^2 + \int\limits_{\re^N} |(-\D)^{-\frac 1 2} u|^2\right)
  + \frac{1}{p+1} \int\limits_{\re^N}     |u+u^*|^{p+1}-  \frac{1}{p+1} \int\limits_{\re^N}
      |u^*|^{p+1} +  \int\limits_{\re^N}  u
      |u^*|^{p}\right|}\\ & \tex{
      \leq K \left( \|u_n\|_{W_{\rm rad}^{1,2}(\re^N)}^2  +\|u^*\|^{p-1}_{L^p} \|u_n\|_{L^{2p}}+ \|u_n\|^{p+1}_{L^{p+1}}\right)\leq K\|u_n\|_{W_{\rm rad}^{1,2}(\re^N)}^2,}
\end{align*}
for a positive constant $K$ and, hence, proving the (PS) condition.
\end{proof}

To conclude the proof of the conditions of the Mountain Pass
Theorem \ref{MPT}, we prove the following result.

\begin{lemma}
\label{leMPT2}
Let the functional $\mc{J} \in C^1(W_{{\rm rad}}^{1,2}(\re^N),\re)$ be the functional denoted by \eqref{funnM}. Then,
\begin{enumerate}
\item[a)] There exist $\rho, \a >0$ such that
$$\mc{J}_{\p B(0,\rho)} \geq \a,$$
where $B(0,\rho)$ represents the ball centered at the origin and of radius $\rho>0$;
\item[b)] also, there exists $e\in W^{1,2}(\re^N) \setminus B(0,\rho)$ such that $\mc{J}(e)\leq 0.$
\end{enumerate}
\end{lemma}
\begin{proof}
First, taking $u\in W_{{\rm rad}}^{1,2}(\re^N)\setminus \{0\}$ and
assuming the Sobolev embedding \eqref{compact},
 we arrive at
     \be
     \label{mpcon1}
     \tex{\mc{J}(u)= K \|u\|_{W_{{\rm rad}}^{1,2}(\re^N)}^2 + o(\|u\|_{W_{{\rm rad}}^{1,2}(\re^N)}^2)
      \quad \hbox{as}\quad u\to 0.}
    \ee
     Then, to obtain such a result we use the inequality \eqref{auxsol} in Lemma\;\ref{leauxz} so that
     \be
     \label{satin}
     \tex{\int\limits_{\re^N} H(u^*,u) - \frac{p}{2}|u^*|^{p-1} u^2 \leq \e \int\limits_{\re^N} |u^*|^{p-1} u^2+ C_\e \|u\|_{L^{p+1}(\re^N)}^{p+1},}
     \ee
     with $C_\e$ a positive constant that depends on $\e>0$.

Moreover, thanks to the non-local weighted eigenvalue problem \eqref{nonlspe}
and the inequality \eqref{satin}
 $$\tex{ \mc{J}(u) \geq \frac{1}{2}\big(1-\frac{1}{\l_1^*}\big)   \|u\|_{W_{{\rm rad}}^{1,2}(\re^N)}^2 -\frac{\e}{p\l^*}
   \|u\|_{W_{{\rm rad}}^{1,2}(\re^N)}^2 - C_\e  \|u\|_{L^{p+1}(\re^N)}^{p+1}.}$$
Consequently, choosing $\e< \frac{p}{2} (\l_1^* -1)$, we finally
have \eqref{mpcon1}. Thus, we can assure that there exists $\rho,
\a >0$ such that $$\mc{J}_{\p B(0,\rho)} \geq \a.$$ Additionally,
assuming the compact Sobolev's embedding \eqref{compact}, as well as the Hardy--Littlewood--Sobolev's inequality \eqref{hls}
  and satisfying the decay condition at infinity \eqref{infcon}, we can see that, applying H\"{o}lder's
  inequality,
\be
\label{minw}
    \begin{split} & \tex{\mc{J}(tu) =
    \frac{t^2}{2} \int\limits_{\re^N} |\nabla u|^2 + \frac{t^2}{2} \int\limits_{\re^N} |(-\D)^{-\frac 1 2} u|^2 -\frac{1}{p+1} \int\limits_{\re^N}
    |tu+u^*|^{p+1}} \\ & \tex{ +  \frac{1}{p+1} \int\limits_{\re^N}     |u^*|^{p+1} +  \frac{t}{p+1} \int\limits_{\re^N}
     u |u^*|^{p+1}  } \\ & \tex{\leq \frac{t^2K}{2} \|u\|^2_{W_{{\rm rad}}^{1,2}(\re^N)}
    -\frac{1}{p+1} \int\limits_{\re^N}     |tu+u^*|^{p+1}+  \frac{1}{p+1} \int\limits_{\re^N}     |u^*|^{p+1} +  \frac{t}{p+1} \int\limits_{\re^N}      u |u^*|^{p+1}
    } \\ & \tex{ \leq \frac{t^2K}{2} \|u\|^2_{W_{{\rm rad}}^{1,2}(\re^N)}
     -\frac{1}{p+1} \int\limits_{\re^N}     |tu+u^*|^{p+1}+ K  \|u^*\|^{p+1}_{W_{{\rm rad}}^{1,2}(\re^N)}}
     \\ & \tex{+
    \frac{t}{p+1}
    \|u^*\|_{L^{p}(\re^N)}^{p-1} \|u\|_{L^{2}(\re^N)}
   \to -\infty,
    }
    \end{split}
    \ee
as $t\to \infty$ and for a positive constant $K$, proving the
second condition of the lemma.

\end{proof}



 \subsection{Expression for the second solution for the equation \eqref{nonloc}}


\noindent Finally, combining  the results obtained by
Lemmas\;\ref{leMPT1} and \ref{leMPT2} with the proof of the (PS)
condition for the functional $\mc{J}$ denoted by \eqref{funnM}, we
ascertain the existence of a non-trivial solution for the equation
\eqref{nonlocM} denoted by $u_0^*$ such that
\begin{align*}
\tex{ \mc{J}(u_0^*) }�& \tex{=\frac{1}{2} \int\limits_{\re^N}
|\nabla u_0^*|^2 + \frac{1}{2} \int\limits_{\re^N} |(-\D)^{-\frac
1 2} u_0^*|^2 -
     \frac{1}{p+1} \int\limits_{\re^N}     |u_0^*+u^*|^{p+1} }\\ & \tex{+  \frac{1}{p+1} \int\limits_{\re^N}
        |u^*|^{p+1} +  \frac{1}{p+1} \int\limits_{\re^N}u_0^*
        |u^*|^{p+1}=c_0,\quad \mbox{with}
    }
\end{align*}
 $$ c_0=\inf_{\g\in\G} \max_{\theta \in [0,1]}
\mc{J}(\g(\theta)), \quad \hbox{where}\quad \G:=\{\g\in
C([0,1],W_{{\rm rad}}^{1,2}(\re^N))\,;\, \g(0)=0,\; \g(1)=tu\},$$
for any $u\in
W_{{\rm rad}}^{1,2}(\re^N)\setminus \{0\}$.  Consequently, by construction of
the equation \eqref{nonlocM}, we obtain a second solution of the
equation \eqref{nonloc} of the form $$u_* = u_0^*+u^*.$$

\begin{remark}
Here we only obtain the existence of two solutions. Moreover, since those solutions could be oscillatory of changing sign
 we cannot actually compare them, nor knowing which ones they are in the subsequent sequence of solutions obtained via
 Lusternik--Schnirel'man's theory.
\end{remark}


\section{Towards to a first countable family of L--S critical
points}
 \label{S4}

\noindent In order to estimate the
number of critical points of a functional, we shall need to apply
 Lusternik--Schnirel'man's (L--S) classic theory of calculus of
variations. Thus, the number of critical points of the functional
\eqref{funn} will also depend on the category of a functional
subset (see below some details).

 There are very important applications of minimax methods to establish the existence of multiple critical points
 of functionals, which are invariant under a group of symmetries $\mc{G}$, in the sense that
 $$\mc{F}(h u)=\mc{F}(u),\quad \forall h\in\mc{G},$$
 and $u\in W^{1,2}(\re^N)$ (a Banach space for our particular case). Note that, normally, the group of symmetries is
 $$\mc{G}:=\{{\rm id}, -{\rm id}\}.$$
 As a simple case, suppose the functional $\mc{F}$ to be even, then $\mc{F}(u)=\mc{F}(-u)$ for any $u$ in the appropriate Banach space.
 One of the most famous methods to get these multiplicity results is due to Lusternik--Schnirel'man for symmetric
 functionals (see \cite[Chapters 8, 9, 10]{Ra2} for further details and \cite{PV} for a discussion of this methodology for a similar problem to the one under consideration here).

 Basically, this topological theory for potential compact operators is a
natural extension of the standard minimax principles which
characterize the eigenvalues of linear compact self-adjoint
operators. Applying the Calculus of Variations to the eigenvalue
problem in an appropriate functional setting, one can see that the
critical values of the functional involved are precisely the
eigenvalues of the problem. Indeed, performing this characterization in the unit sphere $S^{N-1}$ one
has that the eigenvalues are the critical points of the functional associated to a linear operator $L$ in the unit ball
\be
\label{unitb}
\tex{\p \Sigma:=\{v\,:\, \|v\|=1\}.}
\ee
To extend these ideas to nonlinear
potential operators, Lusternik--Schnirel'man introduced the
concept of {\em category} getting a lower estimate  of the number
of critical points on the projective spaces. Indeed, this is estimated by the topological concept of the {\em genus} of a set
introduced by Krasnosel'skii in the 1951 \cite{Kras51}, avoiding the transition to the projective spaces obtained by
identifying points of the sphere which are symmetric with respect to the centre, needed to estimate the category of Lusternik--Schnirel'man. Thus, the
genus of a set provides us with a lower bound of the category.

Moreover, an estimate of the number of critical
points of a functional is at the same time an estimate of the
number of eigenvectors of the gradient functional (in
Krasnosel'skii's terms) and, hence, of the number of solutions of
the associated nonlinear equation.

In our particular case, this functional subset is the following:
\begin{equation}
 \label{R0}
    \tex{ \mc{R}_{0}=\Big\{v\in W^{1,2}(\re^N)\,:\,{\bf H}(v) \equiv \int\limits_{\re^N} |\nabla v|^2 +
     \int\limits_{\re^N} |(-\D)^{-\frac 1 2}v|^2 =1\Big\},}
\end{equation}
in the spirit of the eigenvalues of linear operators in the unit ball \eqref{unitb}.
According to the L--S approach (see \cite{Berger, KrasZ}, etc.),
in order to obtain the critical points of a functional on the
corresponding functional subset, $\mc{R}_{0}$, one needs to
estimate the category $\rho$ of that functional subset. Thus, the
category will provide us with the (minimal) number of critical
points that belong to the subset $\mc{R}_{0}$. Namely, similar to
\cite{PV, GMPSob}, we, formally, may use a standard result:
\begin{lemma}
\label{le37}
The category of the manifold $\mc{R}_0$, denoted by $\rho(\mc{R}_{0})$, is given by the number of eigenvalues
(with multiplicities) of the corresponding linear eigenvalue
problem satisfying:
 \be
 \label{rho1}
  \tex{
\rho(\mc{R}_{0}) = \sharp \{\l_\b >0\}, \quad \mbox{where} } \ee
\be
\label{rho2}
 \tex{
{\bf L} \psi_\b \equiv -\D \psi_\b + (-\D)^{-1} \psi_\b= \l_\b
\psi_\b \,\,\,\,\mbox{in}
 \,\,\,\, \re^N, \quad \lim_{|x|\to \infty}\psi_\b(x)=0.
 }
 \ee
 \end{lemma}

\begin{proof}
 Let $\l_{\b}$ be the $\b$-eigenvalue of the linear bi-harmonic problem \eqref{rho2} such that
 \be
 \label{psi11}
  \psi_\b:=\sum_{k\geq 1} a_k
\hat{\psi}_k,
  \ee
taking into consideration the multiplicity of the $\b$-eigenvalue, under the natural ``normalizing"
constraint $$\sum_{k\geq 1} a_k=1.
  $$
   Here, \ef{psi11} represents
the associated eigenfunctions to the eigenvalue $\l_{\b}$ and
$$\{\hat{\psi}_1,\cdots,\hat{\psi}_{M_\b}\},$$
is a basis of the
eigenspace of dimension $M_\b$.
Moreover, assume a critical point
\be
\label{u88}
u=\sum_{k\geq 1} a_k
\hat{\psi}_k,
\ee
belonging to the functional subset \eqref{R0}  and the eigenspace of dimension $M_\b$, i.e, any real linear combination of orthonormal
 eigenfunctions $\{\hat\psi_\b\}$.

 Thus, substituting \eqref{u88} (since we are looking for solutions of that form) into the equation \eqref{nonloc} and using the expression of the spectral problem \eqref{rho2} yields
 $$\tex{\sum_{k\geq 1} a_k \l_k \hat{\psi}_k - \left(\sum_{k\geq 1} a_k \hat{\psi}_k\right)^{p}=0,}
 $$
 which provides us with an implicit condition for the coefficients $a_k$ corresponding to the critical point \eqref{u88}. Indeed, assuming normalized
 eigenfunctions $\psi_\b$, i.e.,
 $$\int\psi_\b^2=1,$$
  and multiplying by $u$ in \eqref{nonloc} and
 integrating we have that
 $$\tex{\sum_{k\geq 1} a_k^2 \l_k- \int_{\re^N} \left(\sum_{k\geq 1} a_k \hat{\psi}_k\right)^{p+1}=0.}$$
   Furthermore, taking into account that $u\in \mc{R}_{0,\l}$  re-writing down (\ref{R0})
 \be
 \label{RR12}
  \tex{
{\bf H}(u)= \int {\bf L}u\, u, }
 \ee
 and substituting into it any real linear combination of orthonormal
 eigenfunctions $\{\psi_\b\}$ yields 
 \[\sum_{k} a_k^2 \l_k=1.\]
  Therefore, ${\mc R}_0$ contains a sphere of an arbitrary bounded
   dimension. Hence, its category is then infinite.
    \end{proof}

 One can say that the functional
\eqref{funn} is between at least two values, a maximum and a minimum one,
$$c \leq  \mc{F} (u)\leq c^*.$$
 Therefore, having at least two positive critical points for such a functional
 and since the L--S characterisation provides us with a lower bound for solutions, but not exactly how many,
we should not ruled out the situation in which there are infinitely many critical points.

Note that $\rho(\mc{R}_0)$ measures, at least, a lower bound of the total of number of L--S critical points.
 Moreover, thanks to the spectral theory shown above in Section\;\ref{S2.Exp}
  we have the sufficient spectral information about the eigenvalue problem
 \ef{rho2} to obtain a sharp estimate of the category \ef{rho1}.

 \subsection{L--S sequence of critical points}

 Thus, we look for critical values $c_\b$ denoted by
 \begin{equation}
\label{cat} \tex{c_\b := \inf\limits_{A\in \mc{A}_\b}
\sup\limits_{u\in A} \mc{F} (u)} \quad  (\b=1,2,3,...),
\end{equation}
corresponding to the critical points  $\{u_\b\}$
 of the functional $\mc{F}(u)$ \eqref{funn} on the set $\mc{R}_{0}$,
 where
 $$
  \mc{A}_\b :=\{A\,:\, A\subset
\mc{R}_{0},\,\hbox{compact subsets},\quad A=-A\quad
\hbox{and}\quad \rho(A) \geq \b\},
  $$
   is the class of closed
sets in $\mc{R}_{0}$ such that, each member of $\mc{A}_\b$ is of
genus (or category) at least $\b$ in $\mc{R}_{0}$.
The fact that
$\mc{A}=-\mc{A}$ comes from the definition of genus
(Krasnosel'skii \cite[p.~358]{Kras}) such that, if we denote by
$\mc{A}^*$ the set disposed symmetrically to the set $\mc{A}$, $$
  \mc{A}^*=\{ v\,:\, v^*=-v\in \mc{A}\},
  $$
   then, $\rho(\mc{A})=1$
when each simply connected component of the set $\mc{A} \cup
\mc{A}^*$ contains neither of the pair of symmetric points $v$ and
$-v$. Furthermore, $\rho(\mc{A})=\b$ if each subset of $\mc{A}$ can
be covered by, a minimum, $\b$ sets of genus one, and without the
possibility of being covered by $\b-1$ sets of genus one.

Note that just applying the definition of those critical points \eqref{cat} we have the next result.
\begin{lemma}
\label{monoc}(Monotonicity property of the genus)
Let $c_j$ be the critical points defined by \eqref{cat}. Then,
\be
\label{mon37} c_1\leq c_2 \leq \cdots \leq c_\b,
\ee
with $\b$ standing for the category of $\mc{R}_0$.
\end{lemma}
\begin{proof}
Taking $\e>0$, due to
definition of the critical values $c_{\b+1}$, we have that a set
$A_1 \in \mc{A}_{\b+1}$ exists, such that
  $$
  \tex{\sup_{v\in A_1} \mc{F}
(u) < c_{\b+1} + \e.}
$$
 Hence, if $A_1$ contains a subset $A_0 \in
\mc{A}_\b$ such that
  $$
  \tex{\sup_{u\in A_0} \mc{F} (u) \leq
\sup_{u\in A_1} \mc{F} (v)< c_{\b+1} + \e,} \quad \mbox{and}
  $$
 $$
 \tex{c_\b
=\inf_{\mc{A}\in \mc{A}_\b} \sup_{u\in \mc{A}} \mc{F} (u) \leq
\sup_{u\in \mc{A}_1} \mc{F} (u)< c_{\b+1} + \e,}
   $$
then,
$$c_\b < c_{\b+1},$$
which completes the proof.
\end{proof}
   Roughly speaking, since the dimension of the sets
$\mc{A}$ belonging to the classes of sets $\mc{A}_\b$ increases
with $\b$ such that
$$\mc{A}_1\supset \mc{A}_2\supset \ldots \supset \mc{A}_\b,$$
 this guarantees that the critical points delivering
critical values \eqref{cat} are all different.
Hence, to get those
critical values we need to estimate the category $\rho$ of that
set $\mc{R}_{0}$.

Additionally, for this functional
we show the
following particular result (see \cite[Chap.~9]{Ra2} for any further details) which provides us with a countable family of
critical points for the functional $\mc{F}$ \eqref{funn} following
the spirit of the Mountain Pass Theorem.

\begin{theorem}
\label{confam} Let $\mc{F} \in C^1(W_{{\rm
rad}}^{1,2}(\re^N),\re)$ be the functional defined by \eqref{funn}
with $p$ an odd number, $C^1$ and $\mc{F}(0)=0$, such that the
conditions of the Mountain Pass Theorem  proved in
Lemma\;$\ref{leMPT22}$ are satisfied. Then, the functional
$\mc{F}$ possesses
a countable number of critical values.
\end{theorem}

\begin{remark}
Thus, we find a countable family of critical points of the
functional $\mc{F}$ defined by \eqref{cat}
such that $c_\beta=\mc{F}(u_\beta)$ with $u_\beta$ a weak solution
of the problem \eqref{nonloc}. However, we cannot assure how many exactly since so far we only acknowledge the
existence of two solutions obtained through the Mountain Pass arguments performed in the previous section.
Hence, suppose $u_\b$ (in fact a critical point of \eqref{funn}) is the function on which
the subsequent infimum is achieved
\be
\label{infimum} \tex{\inf \mc{F} (u) \equiv \inf \Big(\frac{1}{2}
- \frac{1}{p+1}\int\limits_{\re^N}  |u|^{p+1}\Big) }, \quad
\hbox{with}\quad u_\b \in \mc{R}_{0},\ee having at least $\b$
critical points, in other words $\g(\mc{R}_0)\geq \b$. Now, let us
take a two hump structure (as done in \cite{GMPSob}) $$
 \hat{u}(x)=C[u_\b(x) +u_\b(x+a)],\quad C\in \re, $$
 with sufficiently large $|a|$,
If necessary, we also perform a slight modification of $\hat u(x)$
to have exponentially decay solutions at infinity.
 Thus, since $\hat{u}$ belongs to $\mc{R}_{0}$
 we have
   $$ c_{\hat{v}}=\mc{F} (\hat{v}) = 2
\mc{F} (u_\b)> \mc{F} (u_\b)=c_{u_\b}\equiv c_\b.$$
Observe that
 $$
 \tex{\int\limits_{\re^N} |\nabla u|^2 +
     \int\limits_{\re^N} |(-\D)^{-1/2}u|^2 >0.
 }
    $$
 Thus, for any $\hat{u}=M u_\b$ with $u_\b\in \mc{R}_{0}$, such
that we  have that
\be
\label{beta1}
 \tex{
\big(\int\limits_{\re^N} |\nabla \hat{u}|^2 +
     \int\limits_{\re^N} |(-\D)^{-1/2}\hat{u}|^2\big) =M^2>1,
 }
    \ee
and, hence,
$$
 \mc{F} (\hat{u}) > \mc{F} (u_\b),
$$
 meaning that, in the present case, a two-hump structure cannot be a
 L--S $\b$-solution. In particular for $\b=1$ or $\b=2$ we will have just one or two solutions.

 Actually the ones obtained in Section\;\ref{SM}, which are the only ones we know there existence explicitly.
 \end{remark}

 \section{Numerical analysis in 1D and in the radial geometry}
 \label{S5}

Considering radial geometry as discussed in section 2.1, (\ref{maineq}) takes the form
\begin{eqnarray}
 && u^{(4)} + \frac{2(N-1)}{r} u'''
    + \frac{2(N-1)(N-3)}{r^2} u'' - \frac{(N-1)(N-3)}{r^3} u'
     + u  \nonumber \\
 &&  \hskip 2cm   + p |u|^{p-3} u \left( u u''  + (p-1) \left( u' \right)^2
        + \frac{(N-1)}{r} u u' \right) = 0 .
\label{maineqrad}
\end{eqnarray}
This is considered together with the far-field asymptotic behaviour (\ref{inf1N}), written in the form
\begin{equation}
 \mbox{as $r \to \infty$} \hskip 1cm u \sim k_1 r^{-\frac{(N-1)}{2}} e^{-r/\sqrt{2}} \cos \left( \frac{r-k_2}{\sqrt{2}} \right),
\label{ffrad}
\end{equation}
where we conveniently use the constants $k_{1,2}$ in place of $C_{1,2}$, together with symmetry (\ref{bc1}) or anti-symmetry (\ref{bc111}) conditions at the origin. 

\begin{center}
\begin{figure}[htp]
 \hspace{-1.5cm}
\includegraphics[scale=0.3]{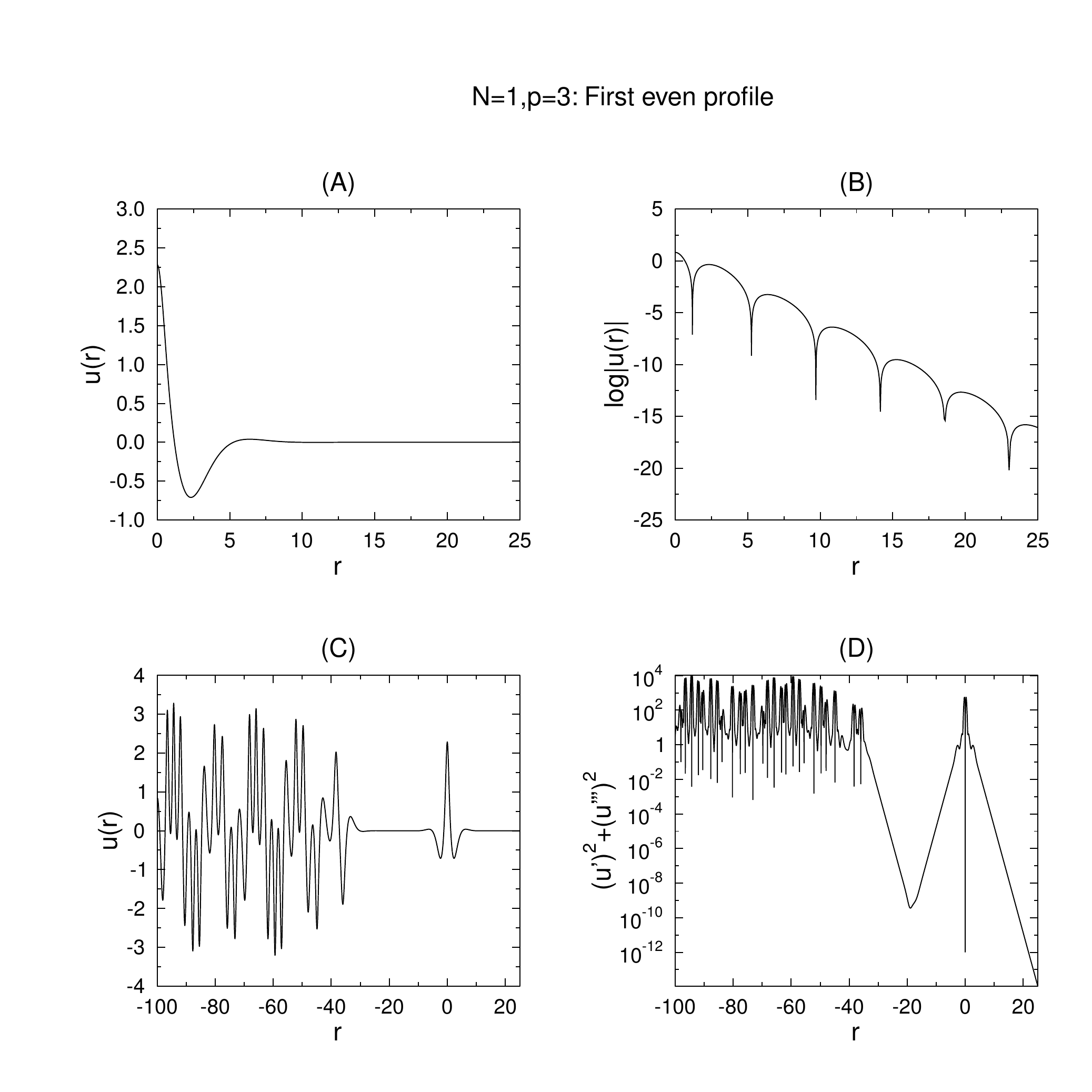}
\vskip 0cm \caption{ \small Numerical illustration of the first even profile for $N=1,p=3$. The profile is
shown in (A), with (B) showing the oscillatory exponentially decaying tail. (C) shows the profile over the extended domain for $r<0$ 
emphasizing the symmetry of the solution. (D) shows the quantity $u'^2+u'''^2$, which
quantifies the accuracy to which the symmetry conditions are satisfied at the origin.
} \label{Fig1}
\end{figure}
\end{center}

Using the far-field behaviour (\ref{ffrad}), a 2-D shooting problem may be formulated where the parameters $k_{1,2}$ are 
determined by satisfying the symmetry condition (\ref{bc1}) at the origin for the even profiles and anti-symmetry condition (\ref{bc111}) 
for the odd profiles. Matlab's ode15s solver is used with tight error tolerances (RelTol=AbsTol=$10^{-13}$).

\begin{center}
\begin{figure}[htp]
 \hspace{-1.5cm}
\includegraphics[scale=0.3]{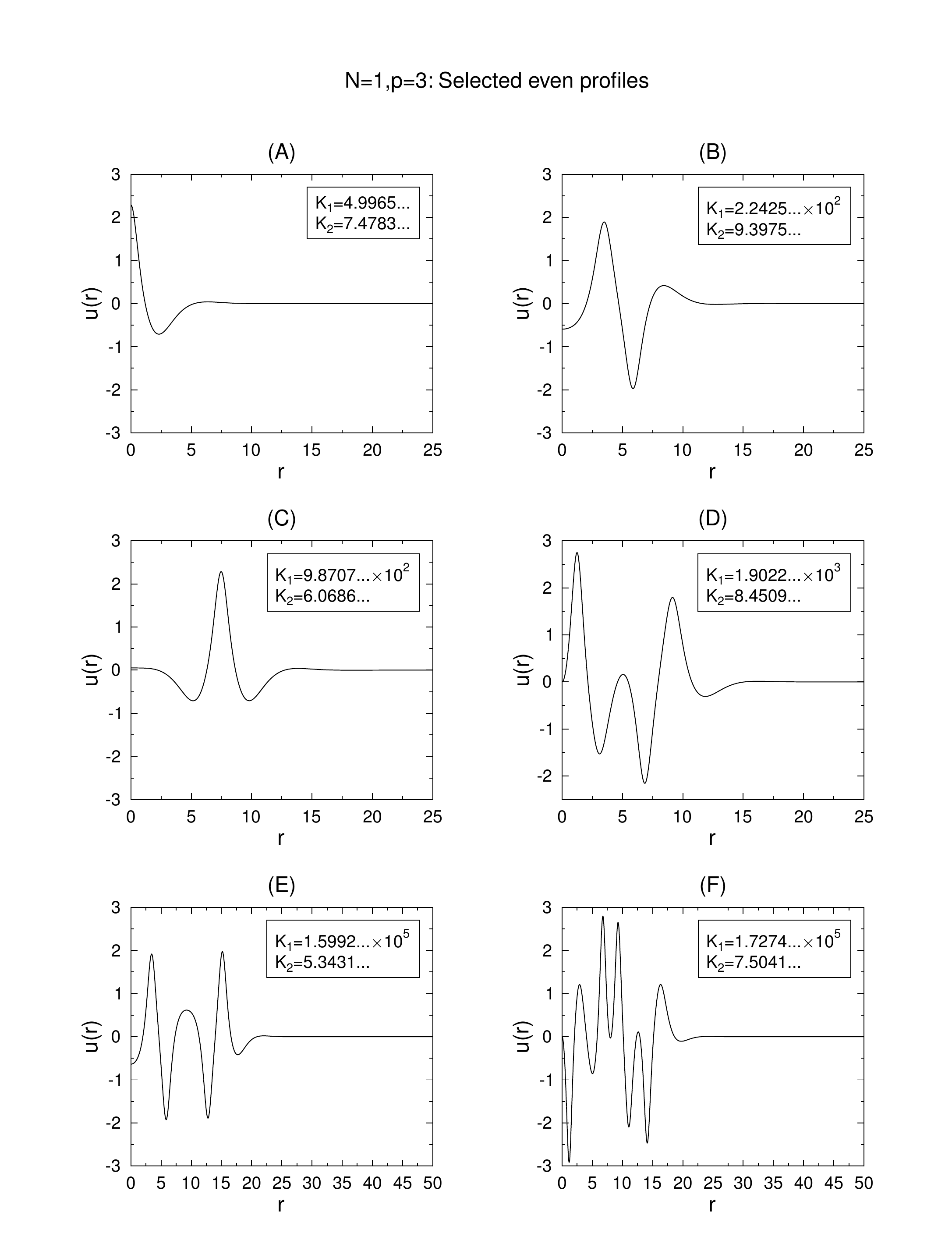}
\vskip 0cm \caption{ \small Illustration of even profiles for $N=1,p=3$.
} \label{Fig2}
\end{figure}
\end{center}

\begin{center}
\begin{figure}[htp]
 \hspace{-1.5cm}
\includegraphics[scale=0.4]{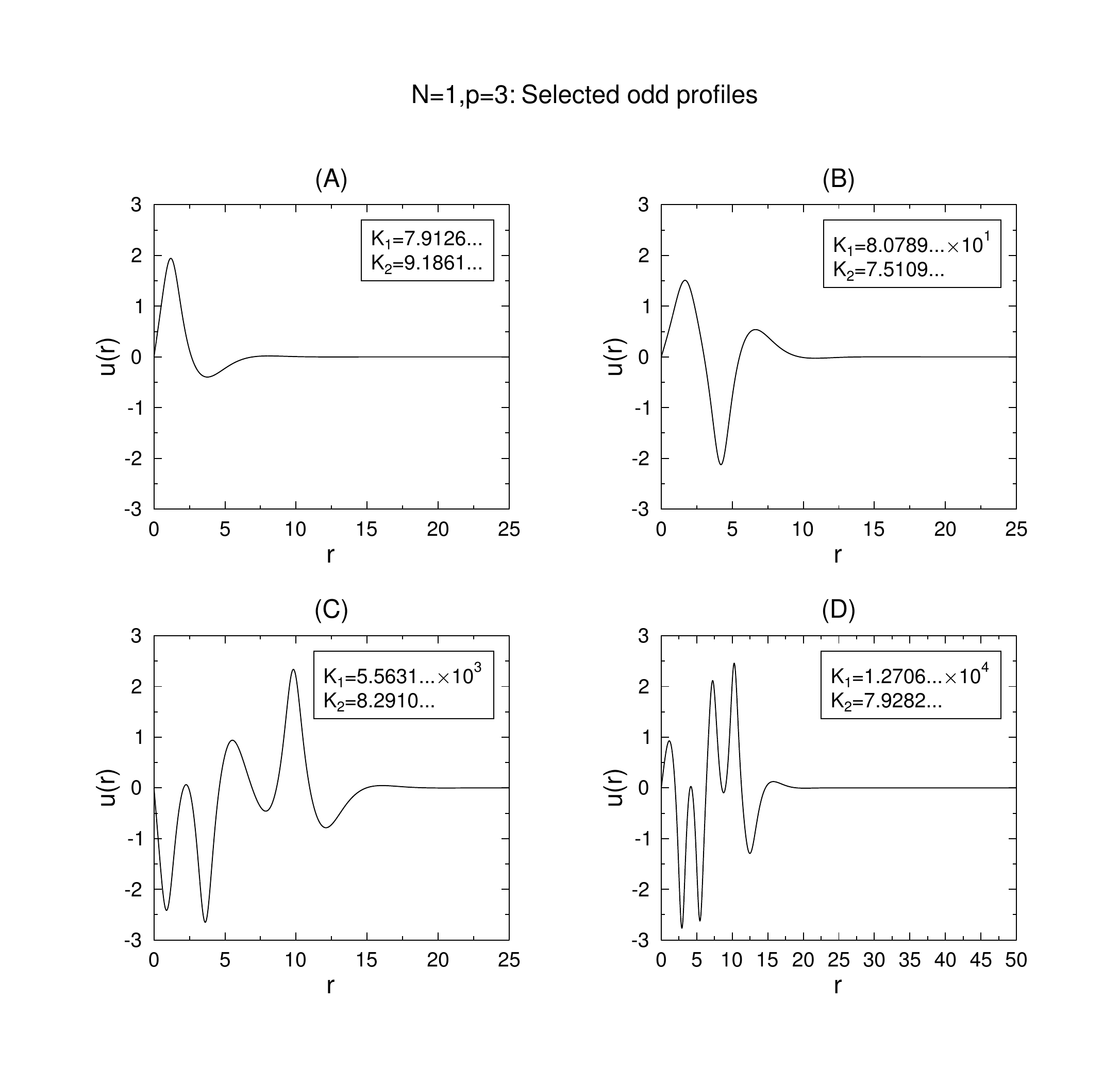}
\vskip 0cm \caption{ \small Numerical illustration odd profiles
for $N=1,p=3$.
} \label{Fig3}
\end{figure}
\end{center}

Figures \ref{Fig1}--\ref{Fig5} 
show illustrative even and odd profiles for $N=1$ and $p=2$ or $p=3$. The condition
(\ref{ffrad}) is used as initial data at sufficiently large $r$ (typically 25 or 50 as given by the domain in the plots).

\begin{center}
\begin{figure}[ht]
 \hspace{-1.5cm}
\includegraphics[scale=0.4]{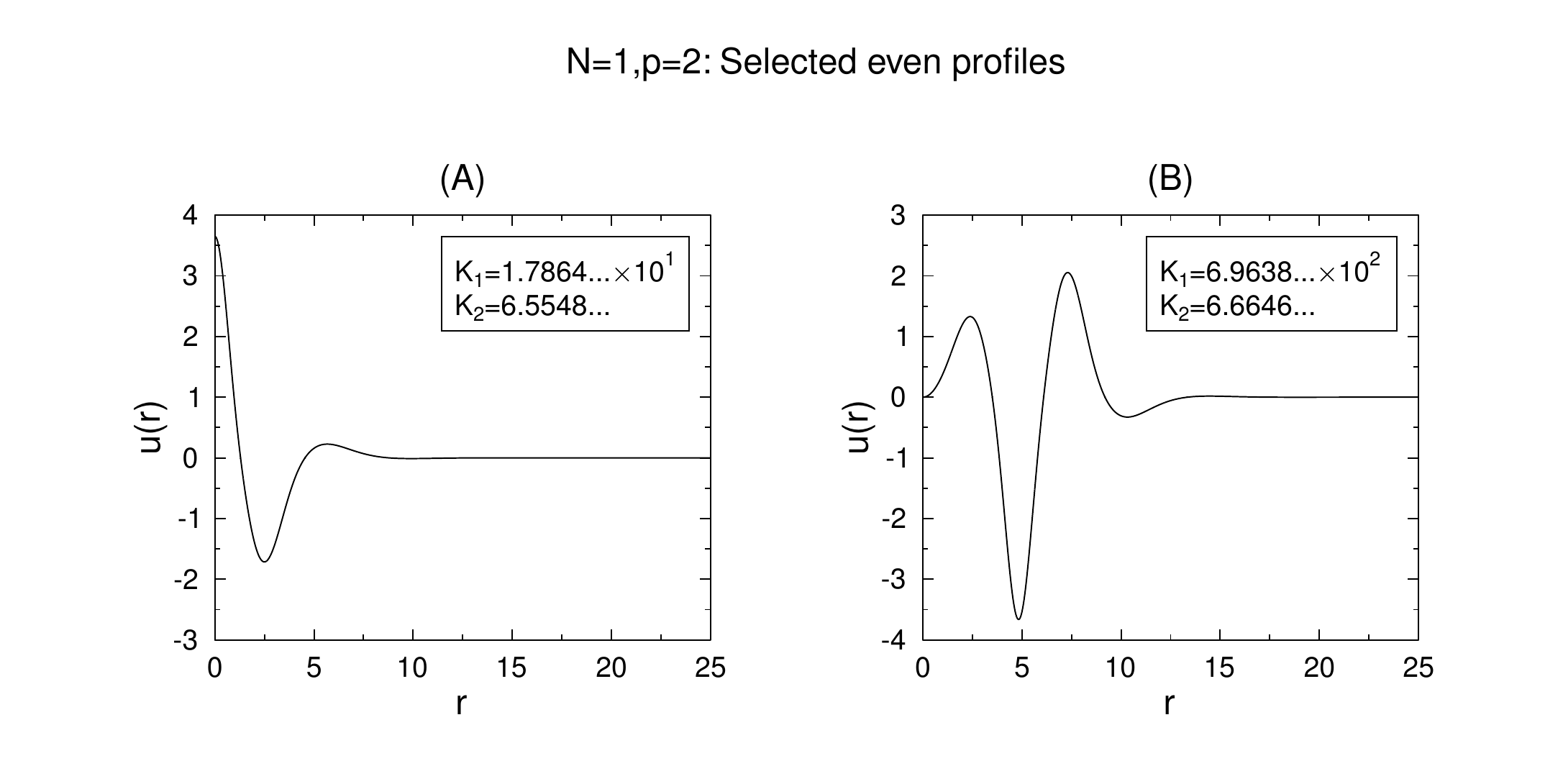}
\vskip 0cm \caption{ \small Numerical illustration of even profiles for $N=1,p=2$.
} \label{Fig4}
\end{figure}
\end{center}

\begin{center}
\begin{figure}[ht]
 \hspace{-1.5cm}
\includegraphics[scale=0.4]{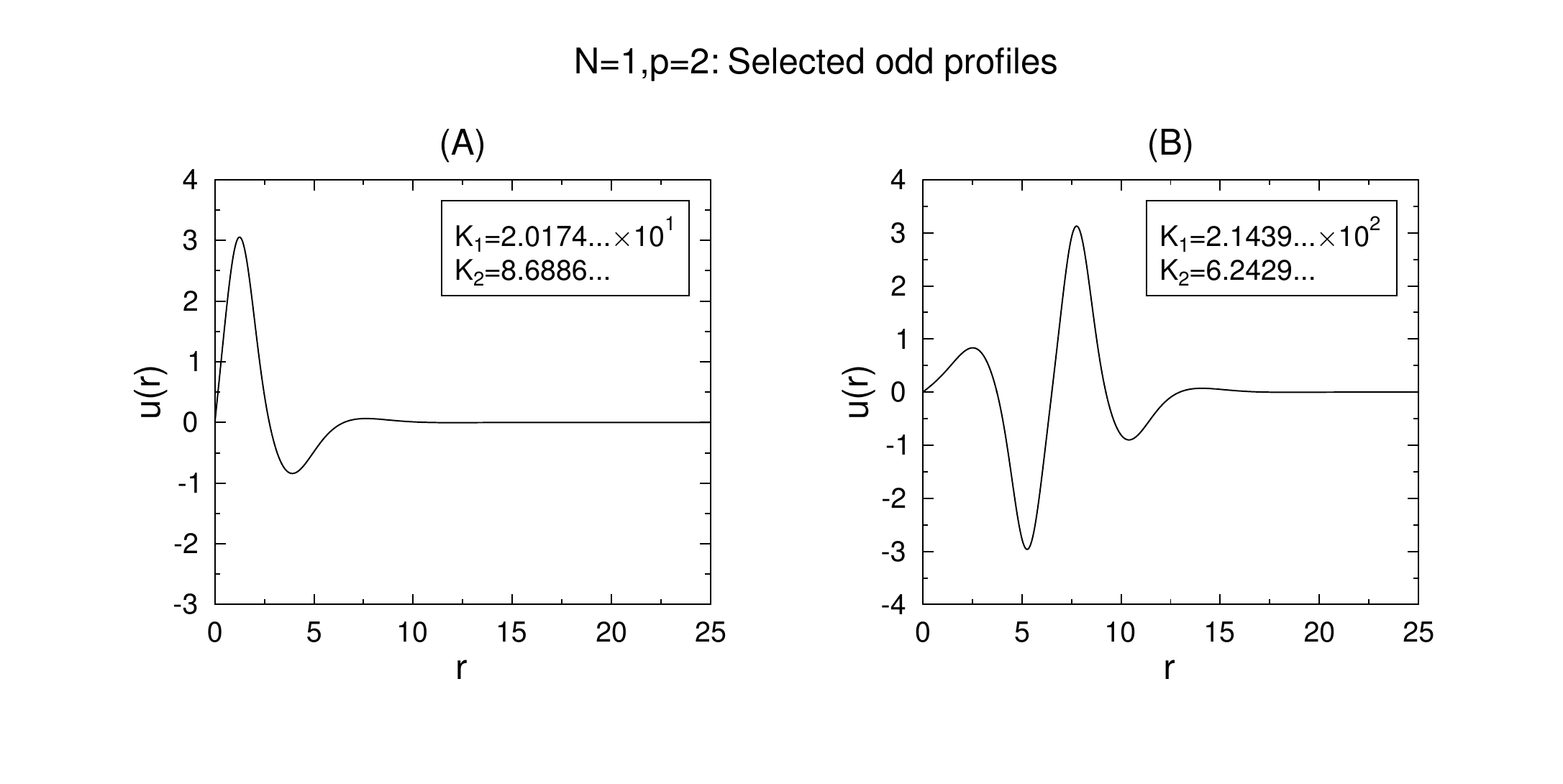}
\vskip 0cm \caption{ \small Numerical illustration of odd profiles for
$N=1,p=2$.
} \label{Fig5}
\end{figure}
\end{center}

We complete this numerical introduction with illustration of a
different class of solutions to (\ref{maineqrad}). 
We may perform
numerical experiments, shooting smoothly from $x=0$ with $u'(0)=u''(0)=u'''(0)=0$ and
varying $u(0)$. 

\begin{center}
\begin{figure}[ht]
 \hspace{-1.5cm}
\includegraphics[scale=0.4]{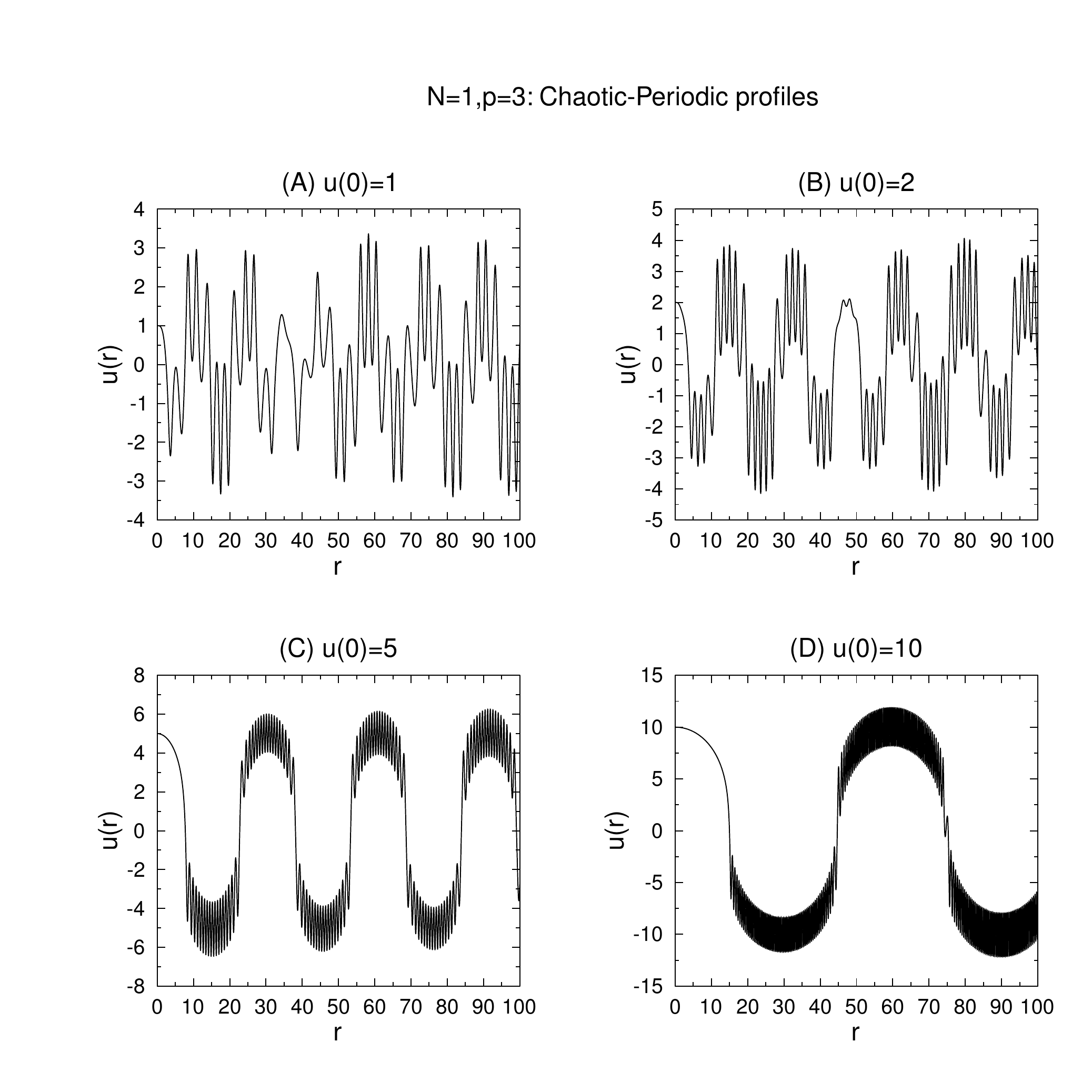}
\vskip 0cm \caption{ \small Numerical illustration of odd profiles for
$N=1,p=2$.
} \label{FigCP3}
\end{figure}
\end{center}

Shown in Figures \ref{FigCP3} and \ref{FigCP2} are selected profiles in one-dimension
$(N=1)$ in the illustrative parameter cases $p=3$ and $p=2$ respectively. For sufficiently small $u(0)$
seemingly chaotic patterns are obtained, which emerge into a more periodic structure as $u(0)$ increases.

\begin{center}
\begin{figure}[ht]
 \hspace{-1.5cm}
\includegraphics[scale=0.4]{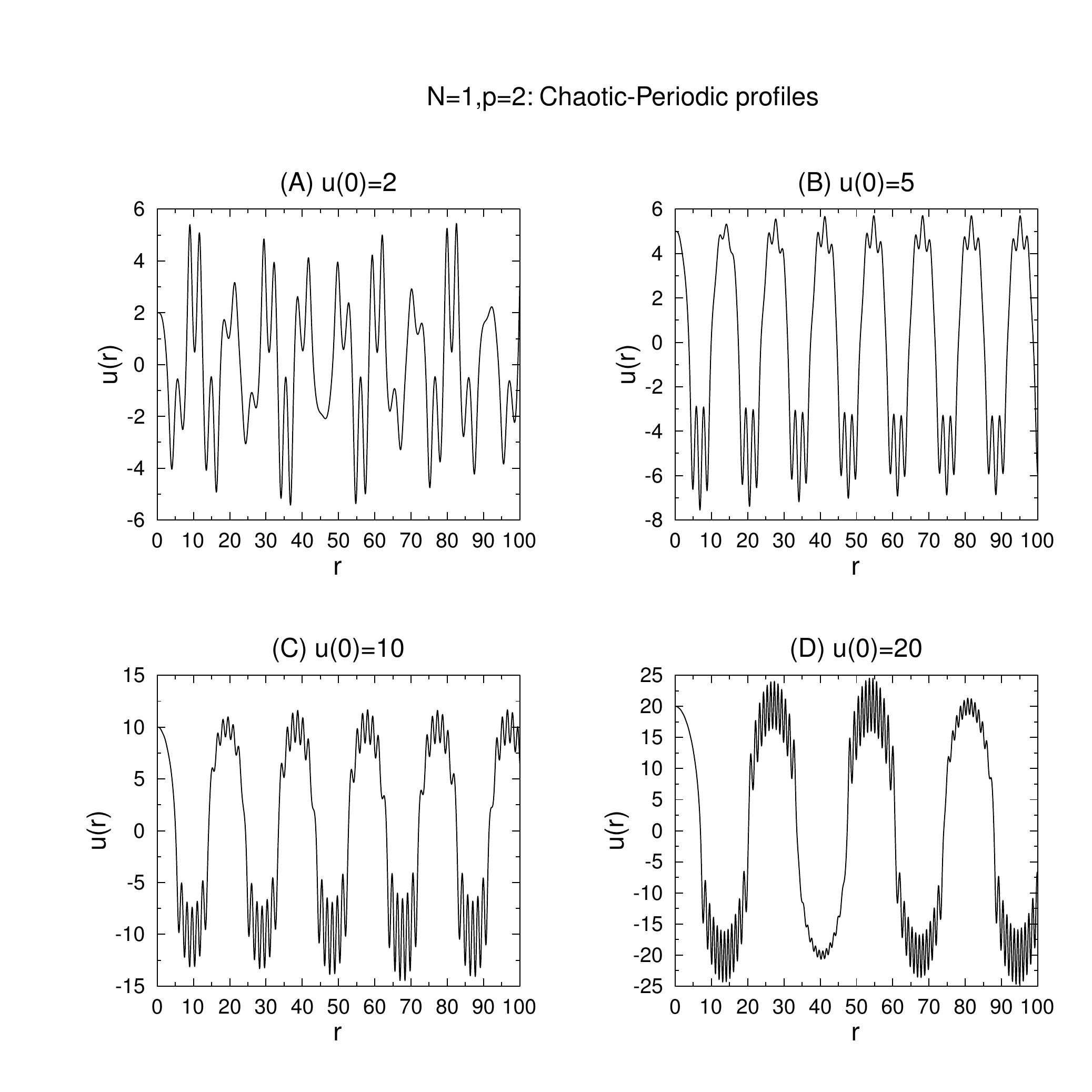}
\vskip 0cm \caption{ \small Numerical illustration of odd profiles for
$N=1,p=2$.
} \label{FigCP2}
\end{figure}
\end{center}

This 
emergence appears sooner for $p=3$ than $p=2$ when increasing the size of $u(0)$. Such transition behaviour is 
seen in similar phase solidification fourth-order equations, such as the Kuromoto-Sivashinsky and 
Swift-Hohenberg equations \cite{PE}, \cite{CG}, as a critical order parameter increases.



\end{document}